\documentclass[12pt]{article}

\input{epsf}
\usepackage{epsfig}
\usepackage{graphics}

\usepackage{graphicx,amsmath,latexsym,enumerate,amsthm,color}
\usepackage{amssymb}
\usepackage{amsmath}
\usepackage[latin1]{inputenc}

\newcommand{\ti}{\tilde}

\newtheorem{proposition}{Proposition}
\newtheorem{remark}{Remark}

\newtheorem{lemma}{Lemma}

\newcommand{\red}{\textcolor{red}}

\begin{document}

\begin{center}
{\bf \huge A hierarchy of models related to nanoflows and surface diffusion}\\
\vspace*{1 cm}
{ Kazuo AOKI, Pierre CHARRIER, Pierre DEGOND}\\
\vspace*{0.5 cm}
Department of Mechanical Engineering and Science\\
Graduate School of Engineering\\
Kyoto University, Kyoto, 606-8501, Japan\\
aoki@aero.mbox.media.kyoto-u.ac.jp\\
\vspace*{0.5 cm}
MATMECA, IMB-Applied Mathematics\\
Universit\'{e} de Bordeaux,\\
33405 Talence cedex, France \\
Pierre.Charrier@math.u-bordeaux1.fr\\
\vspace*{0.5 cm}
MIP, UMR 5640 (CNRS-UPS-INSA),\\
Université Paul Sabatier\\
31062 Toulouse cedex 4, France \\
Pierre.Degond@mip.u-tlse.fr\\
\end{center}

\vspace*{1 cm}

\begin{abstract}    
In last years a great interest was brought to molecular transport problems at nanoscales, 
such as  surface diffusion or molecular flows in nano or sub-nano-channels. In a series of papers 
V. D. Borman, S. Y. Krylov, A. V. Prosyanov and J. J. M. 
Beenakker proposed to use kinetic theory in order to analyze the mechanisms that determine 
mobility of molecules in nanoscale channels. This approach proved to be remarkably useful to
give new insight on these issues, such as density dependence of the diffusion coefficient.
In this paper we revisit these works
to derive the kinetic and diffusion models introduced by V. D. Borman, S. Y. Krylov, A. V. Prosyanov 
 and J. J. M. Beenakker by using classical tools of kinetic theory such as scaling and systematic 
asymptotic analysis. Some results are extended to less restrictive hypothesis.
\end{abstract}


\vspace*{4mm}
\begin{flushright}
to the memory of Carlo Cercignani
\end{flushright}

\section{Introduction}

In last years a great interest was brought to micro and nano-flows, in part driven by applications 
such as  MEMS, micropumps, lab-on-the-chip systems, carbon nanotubes, molecular sieves, etc.
Such flows through micro or nano-geometries (micro- or nano-channels or micro or nano-porous 
materials) exhibit unusual behavior and therefore create the need for new or more precise models 
for numerical simulations. \\

	The gas flows in micro-geometries are often characterized by rather large Knudsen number
but small Reynolds and Mach number. Moreover some phenomena which are usually neglected in classical 
fluid dynamics may take importance, such as thermal creep flow, thermal stress slip flow or thermal 
edge flow. The main
tool at this scale is the kinetic theory used either for direct numerical simulations or for deriving
correct fluid limit models able to take into account the main characteristics of the flow. A survey 
of this approach and a complete bibliography can be found in (\cite{Sone1}) and (\cite{Sone2}) and
for diffusion models on specific geometries and applications the reader can refer to (\cite{AD}), 
(\cite{CD}) and (\cite{A-al}).   
When considering flows through smaller geometries, down to nanoscale, new issues must be addressed 
(see \cite{Beenakker}, \cite{KBA}). The
first one occurs when the size of the channel is comparable with the range of interaction of the gas 
molecules with the wall, i.e., for a diameter of a few nanometers. Then surface-dominated effects become
predominant, and in the vicinity of the surface, the gas flow loses its three-dimensional character 
and becomes two- and, in the limit, one-dimensional, and the mass flux can dramatically exceeds predictions 
of the Knudsen diffusion model (\cite{HPWSAGNB}). Therefore these effects must be included in the models. 
Finally, more complicated 
effects occur when the size of the pore is still smaller (around 0.5 nanometer or smaller), and 
becomes
comparable to the de Broglie wavelength and quantum effects must be considered at this scale. \\  

	In this paper we focus on gas flows through geometries with characteristic size of some
nanometers so that the gas molecule-surface interaction is important but the quantum effects of 
zero-point energy can be neglected. Traditionally, at this scale, for
the numerical simulation, one uses molecular dynamics (\cite{KBA}), or for surface diffusion, lattice 
gas hopping models. Unfortunately molecular dynamics leads to expensive computations and is limited
to short space and time simulations. A different approach is proposed in a series 
of papers (\cite{Borman},\cite{BKP2},\cite{Kry1},\cite{Kry2},\cite{Kry3},\cite{Kry4},\cite{BBK}). 
The authors proposed 
to use kinetic theory in order to analyze the mechanisms that determine 
mobility of molecules in nanoscale channels. This approach proved to be remarkably useful to
give new insight on these issues and, in particular, it was proved that this new kinetic theory of 
surface diffusion is able to explain in a rather natural way  the density dependence of the 
diffusion coefficient which has been observed for instance in zeolites.
Those articles are not enough known in the community of applied mathematics and kinetic theory and it 
seems interesting to revisit them from a mathematical point of view and 
to derive the kinetic and diffusion models by 
using classical tools of kinetic theory such as scaling and systematic asymptotic analysis.
In particular the rigorous approach for modelling molecules trapped by surface effects, introduced in 
(\cite{Degond}) and (\cite{DPV}), will be useful in this context. 
We consider here the simple case of molecules moving in a {\em 2D plane}.\\

The molecule-surface interaction is located on a narrow layer (typically $L=0.3$ nanometers). On microscale 
geometries, for instance in a pipe with diameter around some micrometers the flow of
molecules inside this layer need not to be described precisely and can simply be modelled by a classical 
boundary condition such as specular or diffuse reflection. For  pipes with smaller diameter, about some 
nanometers, but rather long ($X=$ some micrometers or more), the effect of the molecule-surface interaction 
is too much important to be described by a simple boundary condition. So the starting point of our study
is a kinetic model including the molecule-surface interaction through Vlasov terms which takes into account
the potential interaction of the atoms on the surface
and through a relaxation term of molecules by phonons which represents the effect of the thermal fluctuations.
This is a crude representation of the very complex interaction between molecules and the surface but 
sufficient to include in the model much more information than the usual boundary conditions. Nevertheless such 
a model is still very expensive for simulation because of the
smallness of the surface layer and the stiffness of the effect of the interaction potential. It is therefore
useful to look for other models derived from this basic kinetic equation by asymptotic analysis.
 The main influence of the surface on the flow is the confinement of some molecules in a narrow zone close to
the surface, the "surface layer". On the scale we consider here it is relevant to perform
an asymptotic analysis when the ratio $L/X \rightarrow 0$. This leads to the second model of the hierarchy,
which can be seen as a multi-phase model coupling the bulk flow of molecules outside the range of surface 
forces and a two energy group kinetic model describing the surface molecules, i.e., the molecules within 
the range of surface forces (see \cite{Borman}). \\

	However the kinetic equations for the surface molecules still contain several time 
scales that can be quite different and make its solution difficult. The first time scale is given by
$t_{fl}$ the "time of flight" over the potential well, the second one is given by $\tau_{ms}$ the 
molecule-phonon relaxation time and the last one is given by $\tau_z$, the characteristic time for
a molecule to cross the surface layer. These times must be compared with $t_{max}$ the characteristic 
time of observation of the overall system. When $t_{fl}<<t_{max}$, but  $\tau_{ms} \approx t_{max}$, in so 
far as we are interested in the mass flow along the surface  
we do not need to get information on the variation of the distribution function on a short time scale 
($\approx t_{fl}$) and it is useful to derive a kinetic model at a "mesoscopic" time scale comparable to 
$\tau_{ms}$ and $t_{max}$,
by an asymptotic analysis when $t_{fl}/t_{max} \rightarrow 0$. If we consider a larger observation time
such that $\tau_{ms}<< t_{max}$, the system of surface molecules can be described by a diffusion
model obtained by an asymptotic limit of the "multi-phase" kinetic model when $\tau_{ms}/t_{max} \rightarrow 0$.
Thus various diffusion models are derived according to the respective size of $\tau_{ms}$ and $\tau_z$. 
These successive rescalings and asymptotic limits lead to a hierarchy of models which is sketched on 
figure \ref{figpotar}.\\

This paper is organized as follows. The hierarchy of models indicated above is derived in section 2 to 4 
under a simplifying hypothesis on the surface potential and for low density flows for which intermolecular 
collisions are negligible.
In Section 2 we derive a multiphase kinetic model where the equations for the molecules in the range 
of interaction of the surface reduce from two space-dimension  to one-space dimension. 
Section 3 is devoted to the derivation, when $t_{fl}/t_{max} \rightarrow 0$, of a mesoscopic kinetic model 
for the flow of molecules inside the surface layer assuming that the surface potential is rapidly oscillating 
in the x-direction parallel to the surface. This homogenized model is derived under the assumption that surface 
molecules can be described by a one energy group of molecules trapped in the surface layer.
In section 4 we study the diffusion limit,when $\tau_{ms}/t_{max} \rightarrow 0$, of the kinetic model for 
surface molecules derived in section 2. In a  first step we consider that
surface molecules can be described by a one energy group of molecules trapped in the surface layer and
we derive the diffusion limit in the isothermal case and afterwards in the non-isothermal case. In a second 
step the diffusion limit is obtained for the two energy group model in the configuration of a narrow channel
with two surface layers and no bulk flow and we consider several regimes according to the ratio of $\tau_{ms}$ and
$\tau_z$.   

\begin{center}
\begin{figure}
\begin{center}
\includegraphics[height= 8. cm, angle=0]{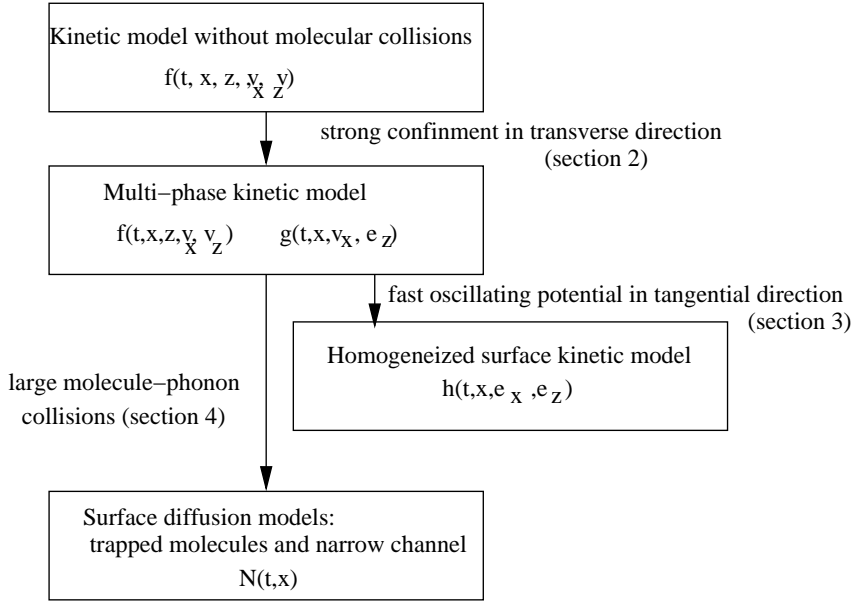}\\
\caption{\em  hierarchy of models.  \label{roadmap} }
\end{center}
\end{figure}
\end{center}

\section{Kinetic model for a gas flow with local interaction with the surface of a solid}

\subsection{Introduction}

In this section we address the issue of modeling a flow of molecules at the vicinity of
a solid wall, for instance in a narrow channel. The influence of the interaction of molecules
with the surface induces a change in the behavior of the gas flow. It has been noticed in
previous works (\cite{Beenakker}, \cite{BBK}) that the flow loses its three dimensional
character and becomes two- and, in the limit, one-dimensional and the transport is 
modified by this phenomenon. The same mechanism has been studied in a more rigorous way
in (\cite{DPV}). Here, since we assume that the molecules are moving in a 2D $(x,z)$ plane, the
flow near the surface will be a one dimensional flow. Following the ideas presented in (\cite{Borman}) and 
(\cite{Beenakker}), and the mathematical approach proposed in (\cite{DPV}), we set up kinetic 
models that take into account the effect of the
molecule-solid interaction and describe the motion of molecules within the range of surface
forces in lower dimension space.\\

  For the sake of simplicity we consider that the molecules move in a plane $(x,z)$ (see remark 1
for the case of a 3D-flow near a plane wall) and we
consider a solid occupying the half-space $z<0$, and a set of molecules of a gas moving in 
the region $z>0$. The state of the gas is described by the distribution function 
$f=f(t,x,z,v_x,v_z)$ and its evolution is modeled by the following kinetic equation
\begin{equation}
\partial_{t}f+v_x\partial_{x}f+v_z\partial_{z}f-\frac{1}{m}\partial_x{\mathcal V}\ \partial_{v_x}f-
\frac{1}{m}\partial_z{\mathcal V}\ \partial_{v_z}f=I_{ph}+Q_{m}, \label{basiceq}
\end{equation}
where ${\mathcal V}={\mathcal V}(x,z)$ is the the interaction potential of the molecules with the 
solid, $I_{ph}$ is the molecule-phonon collision integral describing the interaction of the 
gas molecules with the thermal fluctuations of the solid and $Q_{m}$ describes the 
interaction of the gas molecules with each other. 

We assume that the potential $\mathcal V$ satisfies 
\begin{enumerate}
\item $0 \leq {\mathcal V}$,
\item $\forall x,\ \lim_{z \rightarrow 0}{\mathcal V}(x,z) = + \infty$, 
\item the normal part of the potential is a repulsive-attractive  potential, i.e. 
for every fixed $x$ the potential $z \rightarrow {\mathcal V}(x,z)$ is repulsive (i.e. 
$\partial_z{\mathcal V}(x,z)<0$) for 
$0\leq z < z_m(x)$ and is attractive ($\partial_z{\mathcal V}(z) >0$) for $z_m(x) <z$.
\item The range of the surface forces is finite and thus, the potential satisfies
${\mathcal V}(x,z)={\mathcal V}_m  \mbox{ for}\ z\geq L$. 
\end{enumerate}
\begin{center}
\begin{figure}
    \hfill
    \begin{minipage}{0.45\textwidth}
   \includegraphics[width=0.9\columnwidth,angle=0,clip]{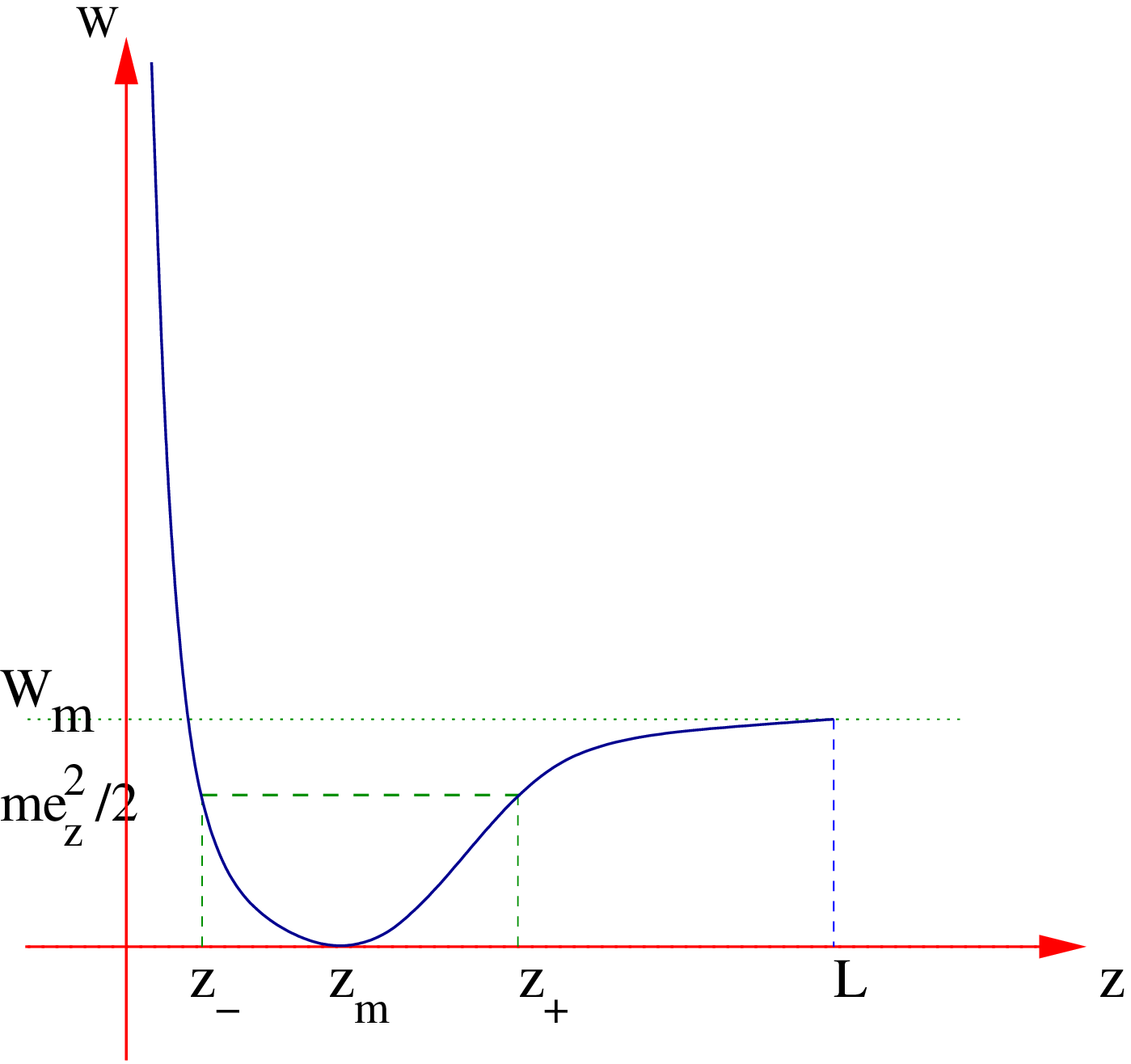}
    \end{minipage}
    \hfill
    \begin{minipage}{0.45\textwidth}
     \includegraphics[width=\columnwidth,angle=0,clip]{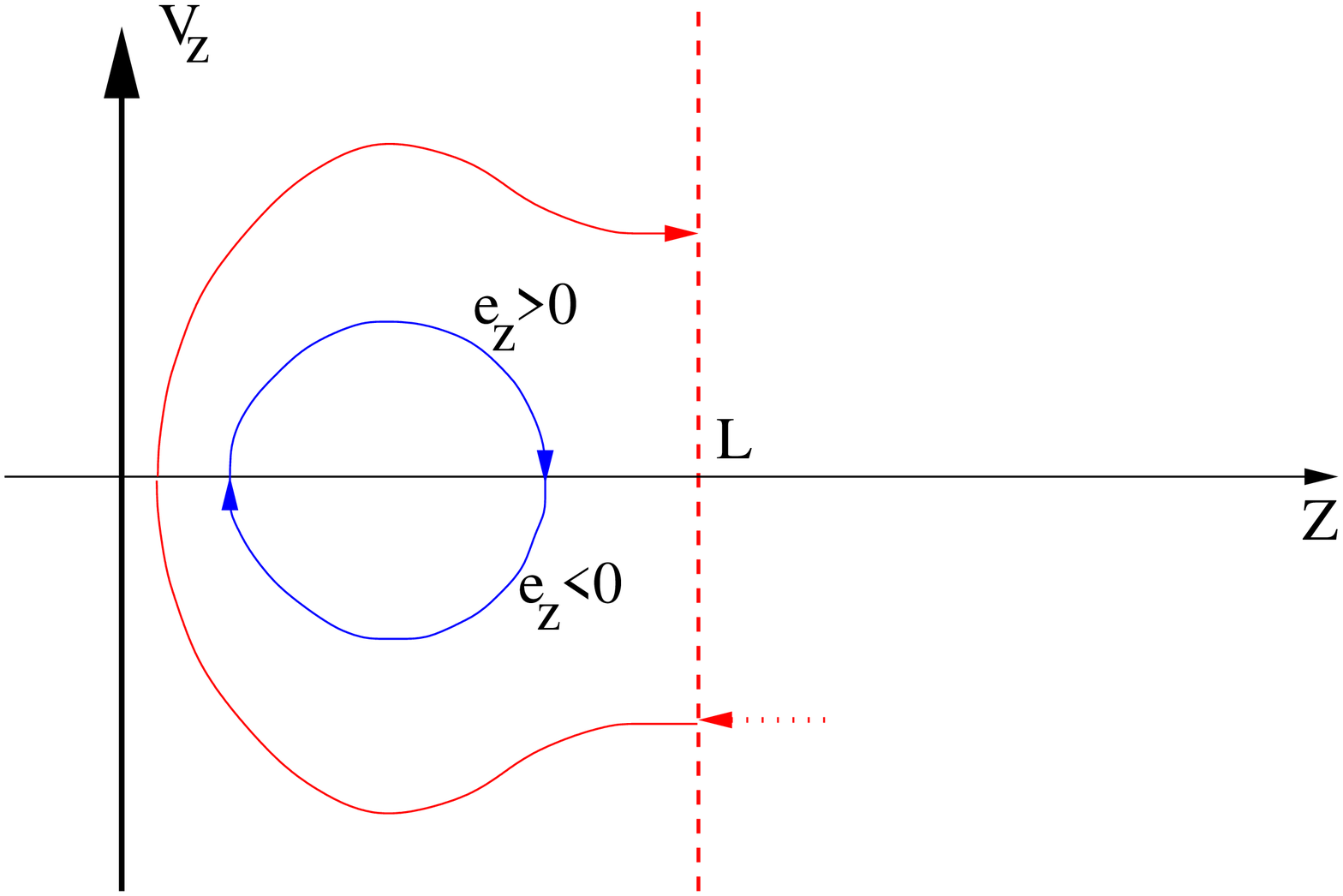}
    \end{minipage}
    \hfill
\caption{\em Repulsive-attractive potential (left) and trajectories of molecules in the phase 
plane (right)
\label{figpotar}  } 
\end{figure} 
\end{center}

Following (\cite{Kry2}), we assume a simplified form of the collision term describing the 
interaction between molecules and phonons. This is a crude approximation but sufficient to 
include in the model new physical effects. More precisely we define
\begin{eqnarray}
\mbox{for}\;  z<L,\; I_{ph}[f]&=& \frac{1}{\tau_{ms}}(\frac{m\ n[f]}{2k\pi T}M-f), \label{Iph}
\end{eqnarray}
where $\tau_{ms}$ is the molecule-phonon relaxation time, $M$ is the dimensionless equilibrium 
distribution function
\begin{equation*}
M(v_x,v_z)=\exp{\left( -\frac{m(v_x^2+v_z^2)}{2kT}\right)},  
\end{equation*}
\begin{eqnarray*}
{n}[f](t,x,z)&=& 
 \int_{v_z} \int_{v_x} f(t,x,z,v_x,v_z)\ dv_x\ dv_z\ ,
\end{eqnarray*}
$m$ is the mass of a gas molecule, $k$ is the Boltzmann constant and $T$ is the temperature. 
Let us remark that the right-hand-side in equation (\ref{basiceq}) satisfies the relation
\begin{equation*}
\int (\frac{m\ n[f]}{2k \pi T}M-f)\ dv =0, 
\end{equation*}		 
which insures the mass conservation.\\

Let us notice that the range of the surface forces $L$ is small 
(typically $L \approx 0.3 nm$) and in most cases much 
smaller than the mean free path of the molecules in the gas phase.
Finally, since the interaction of molecules with phonons 
is important only on a part $[0,L^*]$ of the range L of surface forces ($0 \leq  z \leq L^*$, with 
$L^*<L$), we can assume that $I_{ph}=0,$ for $z\geq L$.\\
Those assumptions imply that in the bulk flow the motion of the gas molecules is given by
the Boltzmann equation
\begin{equation}
\partial_{t}f+v_x\partial_{x}f+v_z\partial_{z}f=Q_{m}, \label{boltzmann}
\end{equation}
which is consistent with the classical kinetic theory.\\

The main goal in our modeling is to take into account the confinement of some gas molecules in the surface layer 
$(0\leq z \leq  L)$ due to the interaction with the surface. But since the width of this surface layer 
is very small, as indicated above, it is useful to derive by asymptotic analysis a model by assuming
that the ratio $L/X$ tends to zero (where $X$ is the characteristic length of observation). The phenomenon
of confinement is mainly transverse to 
the boundary, but since the interaction potential works in both $x$ and $z$-directions the analysis 
is rather technically complicated. We suggest to simplify once more the configuration to make easier the analysis
while keeping in the model enough important features to put in evidence new phenomena near the surface 
at the nanoscale. So we assume that the motion of the gas molecules inside the bulk flow is given by (\ref{boltzmann})
and that the motion of the gas molecule inside the range of the surface forces 
can be separated into two independent components, the tangential and the transverse ones. This can be achieved 
by assuming that
\[ \partial_x{\mathcal V}(x,z)=\varphi(x)  \mbox{ and}\; \partial_z{\mathcal V}(x,z)=\psi(z),\]
which can be obtained by assuming that the interaction potential writes, for $z<L$
\begin{equation}
{\mathcal V}(x,z) = U(x) + W(z), \label{simplepot1}
\end{equation} 
where $W$ is a repulsive attractive potential with $W(z)=W_m$ for $z\geq L$ and where $U$, the tangential component of the potential, satisfies
$0\leq U(x) \leq U_m$. Then if $W_m>>U_m$, we can assume that
\begin{equation}
\mbox{for }\ z>L, \;\;  {\mathcal V}(x,z) = U(x) + W_m \approx W_m, \label{simplepot2}
\end{equation}
which is consistent with the assumption that the range of the surface forces is smaller than $L$ (see figure (\ref{figpotar})).
\\

Another simplification is related to the intermolecular collisions. The study of the influence of
the surface layer on the flow is relevant only in geometries with one characteristic dimension 
(for instance the size of the pore in a porous medium or the diameter of a pipe) about
some nanometers. Thus in a first step it is reasonable to assume that we are working at a scale
where the intermolecular collisions can be neglected so that
\begin{equation}
Q_m=0.\label{nocoll}
\end{equation}

These assumptions (\ref{simplepot1}-\ref{simplepot2}-\ref{nocoll}) are used to derive a hierarchy of models
describing a gas flow in the vicinity of the surface.

\begin{remark}
It should be noted that, under the assumptions (\ref{Iph}) and (\ref{nocoll}), we can
assume that the molecules are moving in the original 3D space
$(x, y, z) (z > 0)$. More specifically, if we assume the 3D version
of (\ref{Iph}) originally for the molecule-phonon interaction, then $f$
corresponds to the marginal velocity distribution function, i.e.,
the integral of the original velocity distribution function with
respect to $v_y$ from $-\infty$ to $\infty$.
\end{remark}

\subsection{Motion of molecules in a repulsive-attractive potential}

Before all we need to introduce some notations associated with the motion of a molecule in a 
one-dimensional repulsive-attractive potential.
 Let us denote by $\varepsilon_z$ the total energy of the normal motion
\begin{equation*}
\varepsilon_{z}= {\frac{1}{2}m v_z^2+W(z)},
\end{equation*}
and let us introduce the "equivalent velocity"
\begin{equation*}
\mbox{for } v_z \neq 0,\;  e_z=e_z(z,v_z)=sgn (v_z)\sqrt{v_z^2+\frac{2}{m}W(z)},
\end{equation*}
so that $e_z$ is an odd function of $v_z$ and
\[ lim_{v_z \rightarrow 0+}\ e_z(z,v_z)= \sqrt{\frac{2}{m}W(z)}  \neq lim_{v_z \rightarrow 0-}\ e_z(z,v_z)
= - \sqrt{\frac{2}{m}W(z)}. \]
We denote $z_m,\  0<z_m<L$, the unique point where $W(z_m)=0$ (see figure (\ref{figpotar})),  and we define 
$z_{-}(e_z)$ and $z_{+}(e_z)$ in the following way: 
\begin{itemize}
\item for  $e_z \neq 0$\\
\begin{eqnarray*}
0<z_{-}(e_z)<z_m<z_{+}(e_z),&& \\
\mbox{for } 0<\sqrt{2W_m/m}<|e_z|,& & z_{+}(e_z)=L, \\
&& W(z_{-}(e_z))=\frac{m}{2}e_z^2, \\ 
\mbox{for } 0<|e_z|<\sqrt{2W_m/m},& & W(z_{+}(e_z))=\frac{m}{2}e_z^2 , \\
&& W(z_{-}(e_z))=\frac{m}{2}e_z^2 .
\end{eqnarray*}
\item for $e_z=0$,
\begin{equation*}
z_{-}(0)=z_m=z_{+}(0)
\end{equation*}
\end{itemize}
We notice that the particles with equivalent velocity $|e_z|< \sqrt{2W_m/m}$ are {\em trapped} (i.e.
cannot leave the "surface layer" $0<z<L$) and move between $z_{-}(e_z)$ and $z_{+}(e_z)$,
but the particles with $|e_z| > \sqrt{2W_m/m}$ are {\em free} and move between $z_{-}(e_z)$ and $z_{+}(e_z)=L$ 
and can go out the surface layer and go into the gas (see fig \ref{figpotar}).
The velocity of a particle with equivalent velocity $e_z$
located at position $z,\ z_{-}(e_z) \leq z \leq z_{+}(e_z)$,  is 
given by
\begin{equation}
v_z(z,e_z)=sgn (e_z)\sqrt{e_z^2-\frac{2}{m}W(z)},\;\;\; \mbox{for}\ z \in [z_{-}(e_z),z_{+}(e_z)],
\label{def-vz}
\end{equation}
and 
\begin{equation}
v_z(z_{-}(e_z),e_z)=v_z(z_{-}(-e_z),-e_z)  =0. \label{cl-}
\end{equation}
Moreover, for trapped molecules we have also
\begin{equation}
v_z(z_{+}(e_z),e_z)=v_z(z_{+}(-e_z),-e_z) =0. \label{cl+}
\end{equation}
Let us define
\begin{equation*}
\sigma_z(z,e_z)= (e_z^2-\frac{2}{m}W(z))^{-1/2}  \mbox{ for}\; |e_z| > \sqrt{2W(z)/m},\\
\end{equation*}
so that $\sigma_z(z,e_z)\ v_z(z,e_z)= sgn (e_z)$, and
\begin{eqnarray*}
{\tau}_z(e_z) &=& \int_{z_{-}(e_z)}^{z_{+}(e_z)} \sigma_z(z,e_z)dz
= \int_{z_{-}(e_z)}^{z_{+}(e_z)}(e_z^2-\frac{2}{m}W(z))^{-1/2}dz, \\
l(e_z)&=& |e_z|\ \tau_z(e_z).
\end{eqnarray*}
As in (\cite{DPV}), ${\tau}_z(e_z)$ can be interpreted as the time
for a trapped molecule to cross the surface layer and $l(e_z)$ is a length.
Moreover, for every fixed $z \in ]0,L]$  the application $v_z \rightarrow  e_z$ is a one-to-one application  from $[0,+\infty[$ onto
$[ \sqrt{\frac{2}{m}W(z)}, +\infty [$ and from (\ref{def-vz}) 
we have
\begin{equation*}
dv_z = |e_z| \sigma_z(z,e_z)de_z \mbox{ for}\  v_z>0,\ e_z>0.
\end{equation*} 
On the same way, for every fixed $z \in ]0,L]$  the application $v_z \rightarrow  e_z$ is a one-to-one
application  from $]-\infty, 0]$ onto $] -\infty,- \sqrt{\frac{2}{m}W(z)}]$ and
\begin{equation*}
dv_z = |e_z| \sigma_z(z,e_z)de_z \mbox{ for}\  v_z<0,\ e_z<0.
\end{equation*} 
Thus if we denote ${\mathcal E}_z(z)= \{e_z,\; |e_z|>(2W(z)/m)^{1/2} \}$, then for a given function $\psi(z,v_z)$ 
we have
\begin{eqnarray*}
\int_{v_z} \psi(z,v_z)\ dv_z&=& \int_{v_z>0} \psi(z,v_z)\ dv_z + \int_{v_z <0} \psi(z,v_z)\ dv_z, \nonumber \\
&=&    \int_{\sqrt{\frac{2}{m}W(z)}}^{+\infty} \psi(z,v_z(z, e_z))\ |e_z| \sigma_z(z,e_z)de_z \nonumber \\
&& + \int_{-\infty}^{-\sqrt{\frac{2}{m}W(z)}} \psi(z,v_z(z, e_z))\ |e_z| \sigma_z(z,e_z)de_z \nonumber \\
&=& \int_{|e_z|>\sqrt{\frac{2}{m}W(z)}} \psi(z,v_z(z, e_z))\ |e_z| \sigma_z(z,e_z)de_z,\nonumber \\
&=& \int_{{\mathcal E}_z(z)} \psi(z,v_z(z, e_z))\ |e_z| \sigma_z(z,e_z)de_z,
\end{eqnarray*}
and
\begin{eqnarray*}
\int_0^L \int_{{\mathcal E}_z(z)}\psi(z,v_z(z, e_z)) |e_z| \sigma_z(z,e_z)\ de_z\ dz \nonumber \\
= \int_{e_z} \int_{z_{-}(e_z)}^{z_{+}(e_z)}\psi(z,v_z(z, e_z))|e_z| \sigma_z(z,e_z)\ dz\ de_z. \label{Fubi-z-e}
\end{eqnarray*}

\vspace*{0.5 cm}
To take into account the molecule-solid interaction we split the flow in two parts, the {\em surface flow}
(for $0 \leq z < L$) and the {\em bulk flow} (for $L < z$). Since the intermolecular collisions are neglected,
 the state of the molecules in the bulk flow is described by the distribution function $f(t,x,z,v_x,v_z)$
which satisfies the following equation
\begin{equation*}
\partial_{t}f+v_x\partial_{x}f+v_z\partial_{z}f = 0.
\end{equation*}

\subsection{ Kinetic model for the molecules inside the surface layer}

We consider now the population of molecules inside the surface layer and we split it into two groups.
The group of "trapped molecules", with $\mid e_z \mid < \sqrt{2W_m/m}$ and the group of "free molecules"
with $\mid e_z \mid > \sqrt{2W_m/m}$.   We introduce for $z_{-}(e_z) \leq z \leq z_{+}(e_z)$,

\begin{eqnarray*}
\phi^t (t,x,z,v_x,e_z)&=&f(t,x,z,v_x,v_z(z,e_z)),\mbox{ for} \mid e_z \mid < \sqrt{\frac{2W_m}{m}},
  \label{defphi-t}\\
\phi^f (t,x,z,v_x,e_z)&=&f(t,x,z,v_x,v_z(z,e_z)),\mbox{ for} \mid e_z \mid > \sqrt{\frac{2W_m}{m}},
  \label{defphi-f}
\end{eqnarray*}
and
\begin{equation*}
\phi(e_z)=\phi^t(e_z)\chi^t(e_z)+\phi^f(e_z)\chi^f(e_z),
\end{equation*}
where ${\chi}^t({e}_z)$ is the characteristic function of the set $\{ {e}_z, |{e}_z|\leq {\sqrt{2{W}_m/m}} \}$ and 
${\chi}^f({e}_z)$ is the characteristic function of the set $\{{e}_z, |{e}_z|> {\sqrt{2{W}_m/m}} \}$.
Let us notice that at the boundary of the surface layer ($z=L$) we have
\begin{eqnarray}
{\phi^f (t,x,L,v_x,e_z)}&=&{f(t,x,L,v_x,v_z(L,e_z))}, \mbox{ for}\ {e_z<0}, \label{condL-}\\
{f(t,x,L,v_x,v_z)}&=& {\phi^f (t,x,L,v_x,e_z(L,v_z))}, \mbox{ for}\ {v_z>0}. \label{condL+}
\end{eqnarray}

Then $\phi^a,\ a=t,f$  satisfy 
\begin{eqnarray}
\partial_t \phi^a + v_x \partial_x \phi^a  -\frac{1}{m}U'(x)\partial_{v_x} \phi^a 
 + v_z(z,e_z) \partial_z \phi^a &=& \frac{1}{\tau_{ms}}(\frac{n[\phi]}{\gamma_0}M-{\phi}^{a}),
\label{eqphi-a}
\end{eqnarray}
where $\tau_{ms}$ is the molecule-phonon relaxation time, $M$ is the dimensionless equilibrium 
distribution function
\begin{equation*}
M(v_x,e_z)=\exp{\left( -\frac{m(v_x^2+e_z^2)}{2kT}\right)},
\end{equation*}
 $n[\phi]$ is the expression of $n[f]$ in term of the new variables
\begin{equation*}
n[\phi] = \int_{v_x} \int_{{\mathcal E}_z(z)} \phi(t,x,z,v_x,e_z)|e_z| \sigma_z(z,e_z)\ de_z\ dv_x,
\end{equation*}
and
\begin{equation*}
\gamma_0 (z) = \int_{v_x} \int_{{\mathcal E}_z(z)} M(v_x,e_z)|e_z| \sigma_z(z,e_z)\ de_z\ dv_x=\frac{2k\pi T}{m}
\exp \left(- \frac{W(z)}{kT}  \right) .
\end{equation*}

Multiplying (\ref{eqphi-a}) by $|e_z| {\sigma_z(z,e_z)}$, we get
\begin{equation}
\partial_t (|e_z|\sigma_z\phi^a) + v_x \partial_x (|e_z|\sigma_z\phi^a)
-\frac{1}{m}U'(x)\partial_{v_x}(|e_z|\sigma_z\phi^a) 
 + e_z \partial_z \phi^a = |e_z|\sigma_z I_{ph}^a[\phi],
\label{eqphi-a2}
\end{equation}
where
\begin{equation*}
I_{ph}^a[\phi] = \frac{1}{\tau_{ms}}\left(\frac{n[\phi]}{\gamma_0}M-\phi^a  \right).
\end{equation*}
Let us notice that, from (\ref{cl-}-\ref{cl+}), we have
\begin{equation}
\phi^t (t,x,z_{\pm}(e_z),v_x,e_z)=\phi^t (t,x,z_{\pm}(e_z),v_x,-e_z)=f(t,x,z_{\pm}(e_z),v_x,0). \label{condz+}
\end{equation} 

We introduce the following dimensionless quantities
\begin{equation} 
\tilde{t}=\frac{t}{t^*},\;{\tilde{\tau}_{ms}}=\frac{\tau_{ms}}{{\tau_{ms}^*}},\;\tilde{x}=\frac{x}{x^*},\;
\tilde{z}=\frac{z}{z^*},\; 
\tilde{e}_{x/z}=\frac{e_{x/z}}{v^*},\; \;
\tilde{v}_x=\frac{v_x}{v^*},\;\tilde{v}_z=\frac{v_z}{v^*}. 
\label{adimvariable*}
\end{equation}
\begin{equation}
 \tilde{f}=\frac{f}{\phi^*},\;\; 
\tilde{\phi}=\frac{\phi}{\phi^*},\;\;  
\ti{\sigma}_z= \frac{\sigma_z}{\sigma^*}, \;\;\tilde{U}=\frac{{U}}{U^*}, \tilde{W}=\frac{{W}}{U^*}  \label{adimvariable2*}
\end{equation}
where  ${\tau_{ms}^*}$ is a reference {{relaxation time}}, $t^*={\tau_{ms}^*}$, $v^*=\sqrt{(2 k T)/m}$ a reference velocity, 
$x^*=v^*{\tau_{ms}^*}$ and $z^*=L$ 
are reference lengths,  $\phi^*=n^*/{v^*}^2$, where $n^*$ is a reference number density,
$U^*=(m/2){v^*}^2$, $\sigma^*=1/v^*$. Moreover we introduce 
\begin{equation*} \ti{n}[\ti{\phi}](\ti{t},\ti{x},\ti{z})= 
\int_{\ti{v}_x} \int_{ \ti{\mathcal E}_z(\ti{z})}  \ti{\phi}(\ti{t},\ti{x},\ti{z},\ti{v}_x,\ti{e}_z)
|\ti{e}_z|\ti{\sigma}_z(\ti{z},\ti{e}_z)\ d\ti{e}_z\ d\ti{v}_x, \; \; 
\end{equation*} 
(where $\ti{\mathcal E}_z(\ti{z})= \{\ti{e}_z,\ |\ti{e}_z|> \sqrt{\ti{W}(\ti{z})} \}$ ) so that
 $\ti{n}[\ti{\phi}]=n[\phi]/n^*$, and
\begin{equation*}
\ti{M}=\exp{(-\ti{v}_x^2-\ti{e}_z^2}),\;\;\ti{\gamma}_0=\frac{\gamma_0}{{v^*}^2}= \pi \exp{(-\ti{W})}.
\end{equation*}
 Finally we assume that
\begin{eqnarray*}
\varepsilon =\frac{z^*}{x^*}<< 1,
\end{eqnarray*}
is a small parameter. Inserting (\ref{adimvariable*}-\ref{adimvariable2*}) in equation 
(\ref{eqphi-a2}) for $a=t,f$, we get
\begin{equation}
\partial_{\ti{t}} (|\ti{e}_z|\ti{\sigma}_z\ti{\phi}^a) + \ti{v}_x \partial_{\ti{x}} (|\ti{e}_z|\ti{\sigma}_z\ti{\phi}^a)
-\frac{1}{2}\ti{U}'(\ti{x})\partial_{\ti{v}_x} (|\ti{e}_z|\ti{\sigma}_z\ti{\phi}^a) + 
\frac{1}{\varepsilon} \ti{e}_z \partial_{\ti{z}} \ti{\phi}^a 
= |\ti{e}_z|\ti{\sigma}_z \ti{I}_{ph}^a[\ti{\phi}], \label{eqphiadim}
\end{equation}
where
\begin{equation*}
\ti{I}_{ph}^a[\ti{\phi}] = \frac{1}{\ti{\tau}_{ms}}
\left(\frac{\ti{n}[\ti{\phi}]}{\ti{\gamma}_0}\ti{M}-\ti{\phi}^a  \right).
\end{equation*}
To derive a "surface equation" we look for a solution of (\ref{eqphiadim}) in the following form
\begin{eqnarray}
\ti{\phi}^a&=&\ti{\phi}^{a,0} + \varepsilon \ti{\phi}^{a,1} + \varepsilon^2 \ti{\phi}^{a,2}  + ...,\label{asymexp-a}
\end{eqnarray}
where the function $\ti{\phi}^{a,i}$ are taken so that (see (\ref{condz+}))
\begin{eqnarray}
\ti{\phi}^{t,i}(\ti{t},\ti{x},\ti{z}_{\pm},\ti{v}_x,\ti{e}_z)& =&
\ti{\phi}^{t,i}(\ti{t},\ti{x},\ti{z}_{\pm},\ti{v}_x,-\ti{e}_z), \label{cond-t} \\
\ti{\phi}^{f,i}(\ti{t},\ti{x},\ti{z}_{-},\ti{v}_x,\ti{e}_z)& =&
\ti{\phi}^{f,i}(\ti{t},\ti{x},\ti{z}_{-},\ti{v}_x,-\ti{e}_z). \label{cond-f}
\end{eqnarray}
Inserting (\ref{asymexp-a}) in (\ref{eqphiadim}), we get\\

\noindent {\em at leading order}
\begin{equation*}
\ti{e}_z \partial_{\ti{z}} \ti{\phi}^{a,0}=0,
\end{equation*}
and thus $\ti{\phi}^{a,0}$ does not depend on $\ti{z}$, i.e.
\begin{equation}
\ti{\phi}^{a,0}=\ti{\phi}^{a,0}(\ti{t},\ti{x},\ti{v}_x,\ti{e}_z). \label{noz}
\end{equation}
This property together with (\ref{cond-t}-\ref{cond-f}) implies that
\begin{equation*}
\ti{\phi}^{a,0}(\ti{t},\ti{x},\ti{v}_x,\ti{e}_z)-\ti{\phi}^{a,0}(\ti{t},\ti{}x,\ti{v}_x,-\ti{e}_z)=0,
\end{equation*}
i.e. implies that $\ti{\phi}^{a,0}$ is an even function of $\ti{e}_z$.\\

\noindent {\em at order + 1}
\begin{equation*}
\partial_{\ti{t}} (|\ti{e}_z|\ti{\sigma}_z\ti{\phi}^{a,0}) + \ti{v}_x \partial_{\ti{x}} (|\ti{e}_z|\ti{\sigma}_z\ti{\phi}^{a,0})
-\frac{1}{2}\ti{U}'(\ti{x})\partial_{\ti{v}_x} (|\ti{e}_z|\ti{\sigma}_z\ti{\phi}^{a,0}) + 
 \ti{e}_z \partial_{\ti{z}} \ti{\phi}^{a,1} 
= |\ti{e}_z|\ti{\sigma}_z \ti{I}_{ph}^a[\ti{\phi^0}], \label{eqphiadim2}
\end{equation*}
and since $\ti{\phi}^{a,0}$ (and also $|\ti{e}_z|\ti{\sigma}^z\ti{\phi}^{a,0}$) is an even function 
in $\ti{e}_z$ we can take the even part and we obtain 
\begin{eqnarray*}
\partial_{\ti{t}} (|\ti{e}_z|\ti{\sigma}_z\ti{\phi}^{a,0}) + \ti{v}_x \partial_{\ti{x}} (|\ti{e}_z|\ti{\sigma}_z\ti{\phi}^{a,0})
-\frac{1}{2}\ti{U}'(\ti{x})\partial_{\ti{v}_x} (|\ti{e}_z|\ti{\sigma}_z\ti{\phi}^{a,0})
 && \nonumber \\
 + \frac{\ti{e}_z}{2} [ \partial_{\ti{z}}\ti{\phi}^{a,1}(\ti{e}_z)- \partial_{\ti{z}}\ti{\phi}^{a,1}(-\ti{e}_z) ]&=&
|\ti{e}_z|\ti{\sigma}_z \ti{I}_{ph}^a[\ti{\phi^0}]
  .\label{eqphio2}
\end{eqnarray*}
Then we integrate with respect to $\ti{z}$ and denoting  $\ti{{l}}=\ti{{l}}(\ti{e}_z)= |\ti{e}_z| \ti{\tau_z}(\ti{e}_z)$
 where $ \ti{\tau_z}(\ti{e}_z) = \int_{\ti{z}_{-}}^{\ti{z}_+} \ti{\sigma}_z(\ti{z},\ti{e}_z)\ d\ti{z}$, we get,  
\begin{eqnarray*}
\partial_{\ti{t}} (\ti{{l}}(\ti{e}_z)\ti{\phi}^{a,0}) + \ti{v}_x \partial_{\ti{x}} (\ti{{l}}(\ti{e}_z)\ti{\phi}^{a,0})
-\frac{1}{2}\ti{U}'(x)  \partial_{\ti{v}_x}(\ti{{l}}(\ti{e}_z)\ti{\phi}^{a,0})
&& \nonumber \\
+ \frac{\ti{e}_z}{2} \ [\int_{\ti{z}_{-}}^{\ti{z}_+} \partial_{\ti{z}} \ti{\phi}^{a,1}(\ti{z},\ti{e}_z)d\ti{z} - 
\int_{\ti{z}_{-}}^{\ti{z}_+}\partial_{\ti{z}} \ti{\phi}^{a,1}(\ti{z},-\ti{e}_z)d\ti{z}]
&=& \ti{Q}_{ph}^{a,*}[\ti{\phi^0}], \label{eqinteg}
\end{eqnarray*}
where 
\begin{eqnarray*}
\ti{Q}_{ph}^{a,*}[\ti{\phi^0}] &=& \frac{{\ti{{l}}}(\ti{e}_z)}{{\ti{\tau}_{ms}}}\left(\ti{\Theta}^*[\ti{\phi^0}] \ti{M}-
  \ti{\phi}^{a,0} \right),\\
\ti{\Theta}^*[\ti{\phi^0}]({\ti{t},\ti{x},\ti{e}_z})&=& \frac{1}{\ti{\tau}_z}\int_{\ti{z}_{-}}^{\ti{z}_+}
 (\ti{n}[\ti{\phi^0}](\ti{t},\ti{x},\ti{z})\ti{\sigma}_z(\ti{z},\ti{e}_z)/\ti{\gamma}_0 (\ti{z}))\ d\ti{z}
\end{eqnarray*}

But we remark that
\begin{eqnarray}
 \frac{\ti{e}_z}{2} \int_{\ti{z}_{-}}^{\ti{z}_+} \partial_{\ti{z}} (\ti{\phi}^{a,1}(\ti{z},\ti{e}_z) - 
 \ti{\phi}^{a,1}(\ti{z},-\ti{e}_z))d\ti{z}&=& \frac{\ti{e}_z}{2}(
[\ti{\phi}^{a,1}(\ti{z}_{+},\ti{e}_z)-\ti{\phi}^{a,1}(\ti{z}_{+},-\ti{e}_z) ]- \nonumber\\
&&[\ti{\phi}^{a,1}(\ti{z}_{-},\ti{e}_z)-\ti{\phi}^{a,1}(\ti{z}_{-},-\ti{e}_z)]), \label{flux}
\end{eqnarray}
Let us first consider the trapped molecules (i.e. $a=t$). From (\ref{cond-t}) we can conclude that this term vanishes. 
Thus, we obtain the following equation (in dimensionless form) for the trapped molecules
\begin{eqnarray*}
\partial_{\ti{t}} (\ti{l}(\ti{e}_z)\ti{\phi}^{t,0}) + \ti{v}_x \partial_{\ti{x}} (\ti{{l}}(\ti{e}_z)\ti{\phi}^{t,0})
-\frac{1}{2}\ti{U}'(x)  \partial_{\ti{v}_x}(\ti{{l}}(\ti{e}_z)\ti{\phi}^{t,0})
&=& \ti{Q}_{ph}^{t,*}[\ti{\phi}^0].  \label{eqphio1-2}
\end{eqnarray*}
We introduce the dimensionless 1D distribution function $\ti{g}$ which is a density number of molecules per $x$-unit 
\begin{eqnarray*}
\ti{g}^t(\ti{t},\ti{x},\ti{v}_x,\ti{e}_z)&=& \ti{l}(\ti{e}_z)\ti{\phi}^{t,0}(\ti{t},\ti{x},\ti{v}_x,\ti{e}_z), \\
\ti{g}^0(\ti{t},\ti{x},\ti{v}_x,\ti{e}_z)&=& \ti{l}(\ti{e}_z)\ti{\phi}^{t,0}(\ti{t},\ti{x},\ti{v}_x,\ti{e}_z)\ti{\chi}^t(\ti{e}_z)
+ \ti{l}(\ti{e}_z)\ti{\phi}^{f,0}(\ti{t},\ti{x},\ti{v}_x,\ti{e}_z)\ti{\chi}^f(\ti{e}_z),
\end{eqnarray*}
{where $\ti{\chi}^t(\ti{e}_z)$ is the characteristic function of the set $\{\ti{e}_z, |\ti{e}_z|\leq\sqrt{\ti{W}_m} \}$ and 
$\ti{\chi}^f(\ti{e}_z)$ is the characteristic function of the set $\{\ti{e}_z, |\ti{e}_z|>\ti{W}_m \}$.} Then the above
equation writes
\begin{eqnarray*}
\partial_{\ti{t}} \ti{g}^{t} + \ti{v}_x \partial_{\ti{x}} \ti{g}^{t}
-\frac{1}{2}\ti{U}'(x)  \partial_{\ti{v}_x}\ti{g}^{t}
&=& \ti{Q}_{ph}^{t}[\ti{g}^0]
,  \label{eqphio1-2.2}
\end{eqnarray*}
where
\begin{eqnarray*}
\ti{Q}_{ph}^{t}[\ti{g}^0] &=& \frac{1}{{\ti{\tau}_{ms}}}\left( \ti{\Theta}[\ti{g}^0] \ti{{l}}(\ti{e}_z)\ti{M}- \ti{g}^{t} \right),\\
\ti{\Theta}[\ti{g}^0](\ti{t},\ti{x},\ti{e}_z)&=& \frac{1}{{\ti{\tau}_z}(\ti{e}_z)}\int_{\ti{z}_{-}}^{\ti{z}_+} (\ti{n}[\ti{g}^0/\ti{l}](\ti{t},\ti{x},\ti{z})
\ti{\sigma}_z(\ti{z},\ti{e}_z)/\ti{\gamma}_0 (\ti{z}))\ d\ti{z}.
\end{eqnarray*}

Let us now consider the "free molecules". We first remark that (from (\ref{asymexp-a}) {and (\ref{noz})})
\[ \ti{{l}}(\ti{e}_z)\ti{\phi}^{f,0}= \ti{{l}}(\ti{e}_z)\ti{\phi}^{f}(1,\ti{e}_z) + {\mathcal O}(\varepsilon).   \]
Moreover the flux term is given by (\ref{flux}), with $a=f$ and $\ti{z}_{+}=1$ in 
dimensionless variables.
As previously the term in $\ti{z}_{-}$ vanishes because of (\ref{cond-f}) but the term in $\ti{z}_{+}=1$ does not vanish
and represents the flux of molecules between the surface layer and the bulk flow (see figure \ref{figpotar}). \\
Then (\ref{flux}) for free molecules writes 
\[ \frac{\ti{e}_z}{2}[\ti{\phi}^{f,1}(1,\ti{e}_z)-\ti{\phi}^{f,1}(1,-\ti{e}_z) ]  \]
which can be equivalently written (since $\ti{\phi}^{f,0}(1,\ti{e}_z)-\ti{\phi}^{f,0}(1,-\ti{e}_z)$=0)
\[ \frac{\ti{e}_z}{2\varepsilon}[\ti{\phi}^{f,0}(1,\ti{e}_z)+ \varepsilon\ti{\phi}^{f,1}(1,\ti{e}_z)
-(\ti{\phi}^{f,0}(1,-\ti{e}_z)+ \varepsilon\ti{\phi}^{f,1}(1,-\ti{e}_z)) ]. \]
But from (\ref{asymexp-a}), we have 
$\ti{\phi}^{f,0}+\varepsilon\ti{\phi}^{f,1} =\ti{\phi}^{f}+ {\mathcal O}(\varepsilon^2)$, so that the flux term writes
\[ \frac{\ti{e}_z}{2\varepsilon}[\ti{\phi}^{f}(1,\ti{e}_z)-\ti{\phi}^{f}(1,-\ti{e}_z) ]+ {\mathcal O}(\varepsilon), \]
and {since $\ti{\phi}^{f,0}=\ti{\phi}^{f}+ {\mathcal O}(\varepsilon)$} we have 
\begin{eqnarray*}
\partial_{\ti{t}} (\ti{{l}}(\ti{e}_z)\ti{\phi}^{f}(1,\ti{e}_z)) + \ti{v}_x \partial_{\ti{x}} (\ti{{l}}(\ti{e}_z)\ti{\phi}^{f}(1,\ti{e}_z))
-\frac{1}{2}\ti{U}'(\ti{x})  \partial_{\ti{v}_x}(\ti{{l}}(\ti{e}_z)\ti{\phi}^{f}(1,\ti{e}_z))&& \nonumber \\
+ \frac{\ti{e}_z}{2\varepsilon}[\ti{\phi}^{f}(1,\ti{e}_z)-\ti{\phi}^{f}(1,-\ti{e}_z) ]
  &=& \nonumber \\
\frac{\ti{{l}}(\ti{e}_z)}{{\ti{\tau}_{ms}}}\left(\ti{\Theta}^*[\ti{\phi}^{t,0}{\ti{\chi}^t}+\ti{\phi}^{f}\ti{\chi}^f] \ti{G}-  
\ti{\phi}^{f}(1,\ti{e}_z) \right)+ {\mathcal O}(\varepsilon), \label{eqinteg-f3}
\end{eqnarray*}

Using (\ref{condL-}) and denoting (with some abuse of notation), 
\[\ti{\phi}^f(\ti{t},\ti{x},\ti{v}_x,\ti{e}_z)=\ti{\phi}^f(\ti{t},\ti{x},1,\ti{v}_x,\ti{e}_z),\] 
the equation of the free molecules finally writes
\begin{eqnarray*}
\partial_{\ti{t}} (\ti{{l}}(\ti{e}_z)\ti{\phi}^{f}) + \ti{v}_x \partial_{\ti{x}} (\ti{{l}}(\ti{e}_z)\ti{\phi}^{f})
-\frac{1}{2}\ti{U}'(\ti{x})  \partial_{\ti{v}_x}(\ti{{l}}(\ti{e}_z)\ti{\phi}^{f})&& \nonumber \\
+ \frac{|\ti{e}_z|}{2\varepsilon}[\ti{\phi}^{f}(|\ti{e}_z|)-\ti{f}(-|\ti{v}_z(1,\ti{e}_z)|) ]
  &=& \nonumber \\
\frac{\ti{{l}}(\ti{e}_z)}{{\ti{\tau}_{ms}}}\left( \ti{\Theta}^*[\ti{\phi}^{t,0}{\ti{\chi}^t}+{\ti{\phi}^{f}}\ti{\chi}^f]  \ti{G}-  
\ti{\phi}^{f} \right)+ {\mathcal O}(\varepsilon). \label{eqinteg-f4}
\end{eqnarray*}
Introducing as previously 
\begin{eqnarray*}
\ti{g}^f(\ti{t},\ti{x},\ti{v}_x,\ti{e}_z)&=& \ti{l}(\ti{e}_z)\ti{\phi}^{f}(\ti{t},\ti{x},\ti{v}_x,\ti{e}_z), \\
\ti{g}(\ti{t},\ti{x},\ti{v}_x,\ti{e}_z)&=& \ti{g}^{t}(\ti{t},\ti{x},\ti{v}_x,\ti{e}_z)\ti{\chi}^t(\ti{e}_z)
+ \ti{g}^{f}(\ti{t},\ti{x},\ti{v}_x,\ti{e}_z)\ti{\chi}^f(\ti{e}_z),
\end{eqnarray*}
this equation writes
\begin{equation*}
\partial_{\ti{t}}\ti{g}^{f} + \ti{v}_x \partial_{\ti{x}} \ti{g}^{f}
-\frac{1}{2}\ti{U}'(\ti{x})  \partial_{\ti{v}_x}\ti{g}^{f}
+ \frac{{1}}{2\varepsilon\ {\ti{\tau}_z}}[\ti{g}^{f}(|\ti{e}_z|)-{\ti{l}(\ti{e}_z})\ti{f}(-|\ti{v}_z(1,\ti{e}_z)|) ]
  = \ti{Q}_{ph}^f[\ti{g}] + {\mathcal O}(\varepsilon), \label{eqinteg-f4-bis}
\end{equation*}
where $\ti{Q}_{ph}^f[\ti{g}]= \frac{1}{{\ti{\tau}_{ms}}}\left( \ti{\Theta}[\ti{g}]\ti{l}(\ti{e}_z)  \ti{M}- \ti{g}^{f} \right) $.
Let us remark that $\ti{\Theta}[\ti{g}^0]
=\ti{\Theta}[\ti{g}]+ {\mathcal O}(\varepsilon)$, 
so that the equation for trapped molecules writes
\begin{eqnarray*}
\partial_{\ti{t}} \ti{g}^{t} + \ti{v}_x \partial_{\ti{x}} \ti{g}^{t}
-\frac{1}{2}\ti{U}'(x)  \partial_{\ti{v}_x}\ti{g}^{t} &=& \ti{Q}_{ph}^t[\ti{g}]
+ {\mathcal O}(\varepsilon),  \label{eqphio1-3}
\end{eqnarray*}
Denoting $\ti{f}^s(\ti{e}_z)=\ti{f}(-|\ti{v}_z(1,\ti{e}_z)|)$  and neglecting ${\mathcal O}(\varepsilon)$ terms in the above 
equations for free and trapped molecules, we finally obtain the following 
system of coupled equations (in dimensionless form)
\begin{eqnarray*}
\partial_{\ti{t}} \ti{g}^{t} + \ti{v}_x \partial_{\ti{x}} \ti{g}^{t}
-\frac{1}{2}\ti{U}'(x)  \partial_{\ti{v}_x}\ti{g}^{t}
&=& \ti{Q}_{ph}^t[\ti{g}], \\
\partial_{\ti{t}}\ti{g}^{f} + \ti{v}_x \partial_{\ti{x}} \ti{g}^{f}
-\frac{1}{2}\ti{U}'(\ti{x})  \partial_{\ti{v}_x}\ti{g}^{f}
+ \frac{1}{2\varepsilon \ti{\tau}_z}[\ti{g}^{f}(|\ti{e}_z|)-\ti{l}\ti{f}^s ]
  &=& \ti{Q}_{ph}^f[\ti{g}].
\end{eqnarray*}
Now we come back to dimensional quantities, noticing that ${\tau_{z}^*=z^*/v^*=\varepsilon \tau_{ms}^*}$  and we define
\begin{eqnarray*}
{g^t(t,x,v_x,e_z)}&=&{\frac{n^*z^*}{{v^*}^2}\ti{g}^t(\ti{t},\ti{x},\ti{v}_x,\ti{e}_z),} \\
{g^f(t,x,v_x,e_z)}&=&{\frac{n^*z^*}{{v^*}^2}\ti{g}^f(\ti{t},\ti{x},\ti{v}_x,\ti{e}_z),} \\
{g(t,x,v_x,e_z)}&=&{g^t(t,x,v_x,e_z)\chi^t(e_z)+g^f(t,x,v_x,e_z)\chi^f(e_z).}
\end{eqnarray*}
We obtain finally the system governing the flow of molecules

\begin{proposition}
Under hypothesis (\ref{simplepot1}-\ref{simplepot2}-\ref{nocoll}), in the limit of 
a small surface layer ($\varepsilon=\frac{L}{x^*} \rightarrow 0$), the flow of 
molecules can be described by the following multi-phase kinetic model  
\begin{eqnarray}
\partial_{t}f+v_x\partial_{x}f+v_z\partial_{z}f &=& 0, \label{beq} \\
\partial_{t}g^t +v_x\partial_{x}g^t
-\frac{1}{m}U'(x)\partial_{v_x}g^t &=& {Q}_{ph}^t, 
\label{kin-surf-trapped3}\\
\partial_{t}g^f +v_x\partial_{x}g^f
-\frac{U'(x)}{m}\partial_{v_x}g^f
- \frac{{l}\ {f}^s- {g}^{f}(|e_z|)}{2{\tau}_z}
&=&{Q}_{ph}^f,\label{kin-surf-free2}   \\
{l}({e}_z)f(t,x,L,v_x,v_z)_{|v_z>0}&=& g^f(t,x,v_x,e_z(L,v_z)), 
\label{cl-coupling}
\end{eqnarray}
where 
\begin{eqnarray*}
Q_{ph}^a&=& \frac{1}{\tau_{ms}}\left(
\Theta[g]{l}({e}_z){M}  -g^{a}   \right),\;\; a=t,f\\
\Theta[g](t,x,e_z)&=&\frac{1}{{\tau}_z(e_z)}
\int_{z_{-}(e_z)}^{z_{+}(e_z)} (n[g/l](t,x,z)\sigma_z(z,e_z) /\gamma_0(z))dz, \\
{f^s(t,x,v_x,e_z)}&=& {f}(t,x,L,v_x,-|v_z(L,e_z)|),\\
l(e_z)&=& |e_z|\ \tau_z(e_z).
\end{eqnarray*}
\end{proposition}

\begin{remark}
\noindent
\begin{enumerate}
\item This model can be interpreted as a multiphase model. The first phase, constituting of the bulk  
flow of molecules outside the range of 
the surface forces, is described by a usual kinetic equation (\ref{beq}). The molecules within the range 
of the surface forces are considered as a separate phase described by a two energy-group kinetic model, 
with the
low energy "trapped molecules" (\ref{kin-surf-trapped3}) and the high energy "free molecules" 
(\ref{kin-surf-free2}). The two groups are coupled by the collision terms. The two phases are 
coupled by the relation (\ref{cl-coupling}) and the last term of the left-hand-side of equation 
(\ref{kin-surf-free2}).
 Let us notice that equations (\ref{kin-surf-trapped3}) and  (\ref{kin-surf-free2}) do not give a precise
description of the flow of the surface molecules with respect to the distance $z$ to the 
surface, but give a relevant information of the flow parallel to the surface, and therefore
will be useful for evaluating the transport flux in this direction. 

\item The distribution function $f(t,x,z,v_x,e_z)$ describes the number density of gas molecules with velocity 
($v_x$,$v_z$) 
in a unit $(x,z)$-area, but the  distribution functions
$g^t(t,x,v_x,e_z)$ and $g^f(t,x,v_x,e_z)$
describe a number density of surface molecules with velocity ($v_x$, $e_z$) per unit x-length.
As a 
consequence $\int_0^L n[\phi](t,x,z)dz $ is the number density of surface molecules per unit x-length. Moreover
we can easily verify that 
\[\int_{e_z}\int_{v_x}Q_{ph}(v_x,e_z)\ dv_x\ de_z  =0,\]
 which ensures the local conservation of mass for surface molecules.

\item In this multiphase model, the "surface" of the solid part is identified to the interface $z=L$ between the
surface layer and the bulk flow. Equation (\ref{beq}) describes the gas flow and equations 
(\ref{kin-surf-trapped3}-\ref{kin-surf-free2}) give a simplified description of what can be interpreted
as a motion of molecules on the "surface". The condition (\ref{cl-coupling}) can be
interpreted as a non local boundary condition  for the bulk flow giving a description of the interaction 
between the gas molecules and the wall which is more detailed than usual local boundary conditions. More precisely, 
some molecules go from the bulk flow into the surface layer, are sent back by the repulsive interaction potential and immediately 
escape from the surface layer into the bulk flow, giving a specular reflexion. On the other hand, some molecules can go from the bulk flow 
into the surface layer, can have a collision with phonons and lose enough transverse energy so that they are trapped 
inside the boundary layer (where they are transported according to (\ref{kin-surf-trapped3})) and, after some time, can gain in a 
collision with phonons enough energy to escape from the surface layer. Since $\tau_{ms}$ is small compared with the characteristic
time of evolution of the bulk flow, this complicated interaction with the solid surface described by the multiphase model 
can be interpreted in a first approximation as a reflexion of the molecule by the boundary.
\end{enumerate}
\end{remark}

\section{1D-mesoscopic kinetic equation for surface molecules}

\subsection{Introduction}
	In section 2 we derived a multiphase model describing the coupling of the bulk flow with the
motion of molecules on the surface. This motion is given by a set of two coupled kinetic equations 
on variable $x$, including a Vlasov term due to the interaction potential parallel to the surface. 
In many applications this potential field is periodic with a small period $2\delta$ that could be
much smaller than the characteristic distance $x^*$, but nevertheless, for consistency with the asymptotic analysis 
leading to theorem 1, we assume that 
\begin{equation}
L<< \delta << x^*. \label{Ldeltaxstar}
\end{equation} 
Some of the surface molecules have a small
tangential energy so that they are {\em bound} in a potential well where they are oscillating very rapidly. 
At the scale $x^*$ they can be considered as motionless. On the contrary surface molecules with a 
high tangential energy are {\em unbound} and can move across the potential wells, but their velocity strongly oscillate. 
In so 
far as we are interested in the mass flow along the surface over a domain much longer than $\delta$ 
we do not need to get information on the distribution function at the "microscopic" scale of a potential 
unit cell. Thus it is useful to derive a kinetic model at a "mesoscopic" scale larger than $\delta$ but 
comparable to $l_{mp}$ the
mean free path between two "collisions" with phonons. At this scale  a kinetic model describing the motion
of the unbound molecules by their average velocity over the potential cells is relevant. In this section we
derive such a model in a simple configuration where we assume that the free molecules in the
surface layer can be neglected so that all surface molecules are trapped. This is justified
when  $W_m>>kT$, which is true for some practical situations,
and only very few molecules can escape from the interaction range of the normal interaction potential 
$W$ so that $n[g^t+g^f] \approx n [g^t]$. In such situations it is reasonable to assume that 
\begin{equation} W_m=+\infty,  \label{Wminf0} \end{equation}
 so that when a molecule enters the surface layer, it cannot escape. Moreover we neglect the
flux of incoming molecules. This model is derived from the kinetic equation (\ref{kin-surf-trapped3}) 
by an asymptotic analysis when $\delta/x^* \rightarrow 0$.  For this asymptotic analysis we
consider the  hypothesis (\ref{simplepot1}-\ref{simplepot2}-\ref{nocoll}), as above. Moreover we assume
that the molecule-phonon relaxation time is constant i.e.
\begin{equation}
 \tau_{ms}= const \label{tauconst}
\end{equation}
and that $U$ (the tangential part of the potential) is periodic with period 
$2 \delta$. More precisely  
\begin{equation}
U(x)=\hat{U}(x/\delta), \label{Uperiodic1}
\end{equation} 
where $\hat{U}(y)$ is a periodic potential with period 2 defined on $[-1,+1]$ and such that $0 \leq \hat{U}(y) 
\leq U_m$ and $\hat{U}(\pm 1)=U_m$.\\

Before to derive such a model we need to introduce some notations related to the motion of a 
molecule in a periodic potential field. We denote
\begin{equation*}
\mbox{for } v_x \neq 0,\;  e_x=e_x(x,v_x)=sgn (v_x)\sqrt{v_x^2+\frac{2}{m}U(x)}.
\end{equation*}
For any fixed $x$ the application $v_x \rightarrow e_x$ is a one to one application from $[0,+\infty[$ onto 
$[\sqrt{2U(x)/m}, +\infty[$,
and also $v_x \rightarrow e_x$ is a one to one application from $]-\infty,0]$ onto $]-\infty,-\sqrt{2U(x)/m}]$ 
so that we have
\begin{equation*}
\mbox{for } |e_x|>\sqrt{2U(x)/m},\  v_x(x,e_x)= sgn (e_x) \sqrt{e_x^2-\frac{2}{m}U(x)}.
\end{equation*}
The jacobian of the application $e_x \rightarrow v_x$ (for $e_x >0$ for instance) is given by
\begin{equation*}
\mbox{for } v_x > 0, \ dv_x=|e_x|\sigma_x(x,e_x)de_x,
\end{equation*}
where
\begin{equation*}
\sigma_x(x,e_x)= (e_x^2-\frac{2}{m}U(x))^{-1/2}, \; \;  \mbox{for } |e_x|>\sqrt{2U(x)/m}.                   
\end{equation*}
and in the same way
\begin{equation*}
\mbox{for } v_x < 0, \ dv_x=|e_x|\sigma_x(x,e_x)de_x.
\end{equation*}
Then denoting 
\begin{equation*}
{\mathcal E}_x(x)= \{e_x,\; |e_x|>(2U(x)/m)^{1/2}   \}, 
\end{equation*}
we have for every integrable function $\psi(v_x)$, 
\[  \int_{v_x}   \psi(x,v_x)  dv_x =  
  \int_{{\mathcal E}_x(x)}  \psi(x,v_x(x,e_x))|e_x|\sigma_x(x,e_x)\  de_x. \] 

The trajectories of the molecules in the $(x,v_x)$ plane
are the level curves of the the total energy $E(x,v_x)=me_x^2/2$ (see fig \ref{pot2}). If the energy of a 
molecule is less than $U_m$, then
its trajectory is a closed curve ({\em bound molecule}), the molecule is trapped in a potential well. 
On the other hand, if its energy is 
larger than $U_m$ its trajectory is an open curve ({\em unbound molecule}) and those unbound molecules
generate a flow in the x-direction.\\

Since the tangential potential $U$ is periodic with a small period $2\delta$,
the velocity $ v_x(x,e_x)= sgn (e_x)\sqrt{e_x^2-\frac{2}{m}U(x)}$ is rapidly oscillating.  
If we look at the surface molecules on a space scale $x^*>>\delta$, the average velocity of the bound molecules 
is equal to zero. On the other hand the average velocity of an unbound molecule can be obtained as follows.
We introduce $-1 \leq y_-(e_x) \leq y_+(e_x)\leq +1$, defined by
\begin{eqnarray}
{\hat{U}}(y_{\pm}(e_x))&=& \frac{m}{2}e_{x}^2,\;\; \mbox{for}\; e_x^2\leq\frac{2}{m}{U}_m  \label{ypmb}, \\
y_{\pm}(e_x)&=& \pm 1,\; \; \mbox{for}\; e_x^2>\frac{2}{m}{U}_m, \label{ypmu}
\end{eqnarray} 
\begin{figure}
\begin{center}
\includegraphics[height= 4. cm, angle=0]{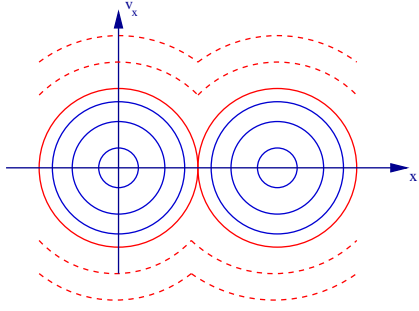}\\
\vspace*{0.3 cm}
\caption{ trajectories in the phase space $(x,v_x)$ for $\hat{U}(y)=k y^2/2.$\label{pot2} }
\end{center}
\end{figure}
and we define for $y \in [-1,+1]$, 
\begin{eqnarray*}
   v_x^{\#}(y,e_x) &=& sgn(e_x) \sqrt{e_x^2-\frac{2}{m}\hat{U}(y)},\;\;\; \mbox{if}\  e_x^2 > \frac{2U_m}{m},   \\ 
  v_x^{\#}(y,e_x) &=& sgn(e_x) \sqrt{e_x^2-\frac{2}{m}\hat{U}(y)},\;\;\; \mbox{if}\ e_x^2 \leq \frac{2U_m}{m}\ \mbox{and}\ y\in [y_-(e_x),y_+(e_x)],\\
  v_x^{\#}(y,e_x) &=&0,\;\;\; \mbox{if}\ e_x^2 \leq \frac{2U_m}{m}\ \mbox{and}\ y\notin [y_-(e_x),y_+(e_x)],\
\end{eqnarray*}
and
\begin{equation*}
\sigma_x^{\#}(y,e_x)=\frac{sgn(e_x)}{v_x^{\#}(y,e_x)},
\end{equation*}
so that $\sigma_x^{\#}(y,e_x)=+\infty$, for $y\notin\ [y_-(e_x),y_+(e_x)]$. {Finally}, we introduce 
$\overline{\sigma}_x(e_x)$ defined by
\begin{eqnarray*}
\overline{\sigma}_x(e_x) &=& \frac{1}{2}\int_{y_-(e_x)}^{y_+(e_x)} \sigma_x^{\#}(y,e_x)dy
=\frac{1}{2} \int_{y_-(e_x)}^{y_+(e_x)}(e_x^2-\frac{2}{m}\hat{U}(y))^{-1/2}dy.
\end{eqnarray*}

Let us consider now an unbound molecule moving in the one dimensional periodic field $U(x)=
\hat{U}(\frac{x}{\delta})$ with an "equivalent velocity" $e_x $.
The position $x$ of the molecule is a monotonic function of $t$, so that we can
consider also $t=t(x)$ as a monotonic function of $x$. Since
\[ dx= (e_x^2-\frac{2}{m}\hat{U}(\frac{x}{\delta}))^{1/2} dt,\]
we have
\begin{equation*}
dt= (e_x^2-\frac{2}{m}\hat{U}(\frac{x}{\delta}))^{-1/2} dx,
\end{equation*}
and the "time of flight", necessary for the considered molecule to cross a potential well (i.e. for going
from $x=(i-1)\delta$ to $x=(i+1)\delta$) is 
\begin{eqnarray*}
\tau_{fl}(e_x) &=& \int_{(i-1) \delta}^{(i+1)\delta}
(e_x^2-\frac{2}{m}\hat{U}(\frac{x}{\delta}))^{-1/2} dx 
= \delta \int_{-1}^{+1}(e_x^2-\frac{2}{m}\hat{U}(y))^{-1/2} dy, \\
&=& \delta \int_{-1}^{+1} \sigma_x^{\#}(y,e_x) dy =  \frac{2 \delta}{|w_x(e_x)|},
\label{meant}
\end{eqnarray*}
where the mean velocity $w_x(e_x)$ of the molecule is given by
\begin{equation*}
{w_x}(e_x) \stackrel{def}{=} \frac{2\ sgn(e_x)\ \delta}{\tau_{fl}(e_x)}.
\end{equation*}
Let us notice that bound molecules, with an "equivalent velocity" $e_x$ with  $me_x^2 \leq 2U_{m}$, are
trapped in a potential well,  between  $(i+y_-(e_x)\delta)$ and $(i+y_+(e_x)\delta)$. Thus for such a molecule,  
the time necessary to cross a potential well can be considered as infinite and by consequence the average velocity 
on a large scale $x^*>> \delta $ can be considered as null, which is consistent with the above definitions of
$\sigma_x^{\#}(y,e_x)$ and of $w_x$. \\

\subsection{Homogenization of the surface kinetic model}
To describe the flow in the {$x$}-direction induced by unbound molecules
we start from the kinetic equation (\ref{kin-surf-trapped3}) for trapped surface molecules, in which we express 
the distribution functions in term of the variable $e_x,\ e_z$, i.e. 
we introduce $h^t$ given  by
\begin{equation*}
   h^t(t,x,e_x,e_z)\stackrel{\rm def}{=}g^t(t,x,v_x(x,e_x),e_z). \label{dfh}
\end{equation*}
 Moreover we denote 
$h^t_c(t,x,e_x,e_z)=h^t(t,x,e_x,e_z)\chi_c(e_x),\ c=b,u$,
where $\chi_b(e_x)$ is the characteristic function of the set $\{e_x,\ |e_x|\leq \sqrt{2U_m/m}\}$ and 
$\chi_u(e_x)$ is the characteristic function of the set $\{e_x,\ |e_x| > \sqrt{2U_m/m}\}$.
From the definition of $y_{\pm}(e_x)$ (\ref{ypmb}-\ref{ypmu}), it comes for bound molecules 
\begin{equation}
h^t_b(t,\delta y_{\pm}(e_x),e_x,e_z)=h^t_b(t,\delta y_{\pm}(-e_x),-e_x,e_z)=g^t(t,\delta y_{\pm}(e_x),0,e_z). 
\label{trappedper}
\end{equation}

Let us recall that the evolution of the surface molecules is described by equations (\ref{kin-surf-trapped3}) which can 
be written
\begin{eqnarray*}
\partial_{t}g^t +v_x\partial_{x}g^t -\frac{1}{m}U'(x)\partial_{v_x}g^t
&=&{Q}_{ph}^t,\label{kin-surf-free3} 
\end{eqnarray*}

From the definition of $h^t$, we get 
$\partial_t {h^t} = \partial_t g^t$, and $\partial_x {h^t} = 
\partial _x g^t+ \partial_{v_x}g^t \partial_xv_x$, 
or
$\partial_x {h^t}= \partial_x g^t -\frac{U'(x)}{mv_x(x,e_x)}\partial_{v_x}g^t
$, 
so that
\begin{equation*}
v_x(x,e_x)\partial_x {h^t} = v_x(x,e_x)\partial_x {g^t}
-\frac{1}{m}U'(x) \partial_{v_x}g^t, 
\end{equation*}
and finally the distribution function $h^t(t,x,e_x,e_z)$  satisfies the following kinetic equation 
\begin{eqnarray}
\partial_t h^t+ v_x(x,e_x)\partial_x h^t &=& 
\frac{1}{\tau_{ms}}\left(\Theta[h^t]{l}(e_z)M-h^t \right), \label{eqht} 
\end{eqnarray}
where
\begin{eqnarray*}
\Theta[h]&=&\frac{1}{{\tau}_z(e_z)}\int_{z_{-}(e_z)}^{z_{+}(e_z)} (n[h/l]\sigma_z(z,e_z) /\gamma_1(x,z))dz,\\
n[\varphi] &=& \int_{{\mathcal E}_z(z)} \int_{{\mathcal E}_x(x)} \varphi(t,x,e_x,e_z) |e_x|\sigma_x(x,e_x)|e_z|\sigma_z(z,e_z)\ 
de_x\ de_z,\\
M(e_x,e_z)&=&e^{-m(e_x^2+e_z^2)/(2kT)}, \\
\gamma_1(x,z)&=&\frac{2\pi kT}{m}\exp \left(- \frac{U(x)+W(z)}{kT} \right).
\end{eqnarray*}

To derive a kinetic
model at a "mesoscopic scale" (i.e. on a characteristic length $x^*>> \delta$), we introduce the following 
dimensionless quantities
\begin{equation} 
\tilde{x}=\frac{x}{x^*},\; \tilde{l}=\frac{{l}}{z^*},\; \tilde{v}_x=\frac{v_x}{v^*},\; \tilde{e}_{x/z}=\frac{e_{x/z}}{v^*},\;  
\tilde{t}=\frac{t}{t^*},\;  \ti{\tau}= \frac{\tau_{ms}}{t^*},\;
\; \ti{n}=\frac{n}{n^*}, \;
\label{adimvariable1}
\end{equation}
and 
\begin{equation} 
\ti{\Theta}=\frac{\Theta\ {{v^*}^2}}{n^*}, \;
 \tilde{h}^t=\frac{h^t}{h^*},\; \ti{\sigma}_x= \frac{\sigma_x}{\sigma^*},\;\ti{\sigma}_x^{\#}= \frac{\sigma_x^{\#}}{\sigma^*},\;
 \tilde{U}=\frac{\hat{U}}{U^*},\;\; \tilde{W}=\frac{\hat{W}}{U^*},
\label{adimvariable2}
\end{equation}
where $x^*$ is a reference length, $v^*=\sqrt{2kT/m}$, $t^*= x^*/v^*$ is a reference time,   
$U^*=(m/2){v^*}^2$, $n^*$ is a reference number density, $h^*=(n^* z^*)/{v^*}^2$, $\sigma^*=1/v^*$,  so that
\[ \tilde{v}_x(\ti{x},\ti{e}_x)=sgn(\tilde{e}_x)\sqrt{\tilde{e}_x^2-\tilde{U}(\tilde{x}x^*/\delta)}. \]
Moreover we assume that the reference length $x^* >> \delta$, so that the ratio  
\begin{equation*}
\varepsilon = \frac{\delta}{x^*} << 1
\end{equation*}
is a small parameter. We now derive an asymptotic model in the limit $\varepsilon \rightarrow 0$, and we describe 
the asymptotic expansion for the trapped molecules.  \\

Inserting (\ref{adimvariable1}-\ref{adimvariable2}) in the kinetic equation (\ref{eqht}) we obtain
\begin{equation*}
\partial_{\tilde{t}}(\tilde{h}^t)+ sgn(\tilde{e}_x)
\sqrt{\tilde{e}_x^2-\tilde{U}(\tilde{x}/\varepsilon)}\partial_{\tilde{x}}(\tilde{h}^t) =
 \frac{1}{\ti{\tau}}\left(\ti{\Theta}[\ti{h}^t]\ti{l}\ti{M}-\ti{h}^t  \right), \label{adimeqht}
\end{equation*}
and we look for a solution in the following form
\begin{eqnarray}
\tilde{h}^t(\tilde{t},\tilde{x},\frac{\tilde{x}}{\varepsilon} ,\tilde{e}_x,\tilde{e}_z)=
 \tilde{h}^{t,0}(\tilde{t},\tilde{x},
\frac{\tilde{x}}{\varepsilon},\tilde{e}_x,\tilde{e}_z)+ \varepsilon 
\tilde{h}^{t,1}(\tilde{t},\tilde{x},\frac{\tilde{x}}{\varepsilon},\tilde{e}_x,\tilde{e}_z)+... \label{exph}
\end{eqnarray}
where  $\ti{h}^{t,i}(\ti{t},\ti{x},y,\ti{e}_x,\ti{e}_z), \; i=0,1,...$ are periodic functions in $y$ with period $2$ and
where, for bound molecules (to be consistent with (\ref{trappedper}) and the definition of $y_{-}(e_x)$ and $y_{+}(e_x)$) 
\begin{equation}  
\ti{h}^{t,i}_b(\ti{t},\ti{x},y_{\pm}(\ti{e}_x),\ti{e}_x,\ti{e}_z)
= \ti{h}^{t,i}_b(\ti{t},\ti{x},y_{\pm}(-\ti{e}_x),-\ti{e}_x,\ti{e}_z).
\label{trappedperadim}
\end{equation}

Then ${\tilde{h}^{t}}$ satisfies
\begin{equation}
\partial_{\tilde{t}}({\tilde{h}^{t}})+ sgn(\tilde{e}_x)
\sqrt{\tilde{e}_x^2-\tilde{U}(\tilde{x}/\varepsilon)}\ [\partial_{\tilde{x}}+\frac{1}{\varepsilon}\partial_y]({\tilde{h}^{t}}) =
 \frac{1}{\ti{\tau}}\left(\ti{\Theta}[{\ti{h}^t}]\ti{l}\ti{M}-\ti{{h}}^t  \right), \label{adimeqhty}
\end{equation}
where
\begin{eqnarray*}
\ti{\Theta}[{\ti{h}}^t]&=&\frac{1}{\ti{{\tau}}_z}\int_{\ti{z}_{-}(\ti{e}_z)}^{\ti{z}_{+}(\ti{e}_z)} (\ti{n}[{\ti{h}}^t/\ti{l}]\ \ti{\sigma}_z(\ti{z},\ti{e}_z) /\ti{\gamma_1}({\ti{x}/\varepsilon},\ti{z}))
d\ti{z},\\
\ti{n}[\ti{\varphi}^*] &=& \int_{\ti{\mathcal E}_z(\ti{z})} \int_{\ti{\mathcal E}_x^{\#}(y)}  \ti{\varphi}^*(\ti{t},\ti{x},y,\ti{e}_x,\ti{e}_z) 
|\ti{e}_x|\ti{\sigma}_x^{\#}(y,\ti{e}_x)
|\ti{e}_z|\ti{\sigma}_z(\ti{z},\ti{e}_z)\  d\ti{e}_x\ d\ti{e}_z. \;  \\
\ti{\mathcal E}_z(\ti{z}) &=& \{\ti{e}_z, |\ti{e}_z|>{\sqrt{\ti{W}(\ti{z})}}   \},\\
\ti{\mathcal E}_x^{\#}(y) &=& \{\ti{e}_x, |\ti{e}_x|>{\sqrt{\ti{U}(y)}}  \}\\
\ti{\gamma}_1(y,\ti{z}) &=& \pi \exp \left(- \ti{U}(y)-\ti{W}(\ti{z})  \right) .  
\end{eqnarray*}
Moreover since
\begin{eqnarray}
\tilde{v}_x(\tilde{x},\tilde{e}_x)&=& sgn(\tilde{e}_x)
\sqrt{\tilde{e}_x^2-\tilde{U}(\tilde{x}/\varepsilon)} =
\tilde{v}_x^{\#}(\frac{\tilde{x}}{\varepsilon},\tilde{e}_x),\label{expu}
\end{eqnarray}
inserting (\ref{exph}, \ref{expu}) in (\ref{adimeqhty}) and balancing order by order in $\varepsilon$ we get
at the principal order,
\begin{equation*}
\tilde{v}_x^{\#}\partial_y\tilde{h}^{t,0} =0,
\end{equation*}
and thus for molecules with $\tilde{v}_x^{\#}\neq 0$
we conclude that
\begin{equation}
\tilde{h}^{t,0}(\tilde{t},\tilde{x},y,\tilde{e}_x,\tilde{e}_z)=
\tilde{h}^{t,0}(\tilde{t},\tilde{x},\tilde{e}_x,\tilde{e}_z). \label{po}
\end{equation}
Then from this property and from (\ref{trappedperadim}) we deduce that $\tilde{h}^{t,0}_b(\tilde{t},\tilde{x},\tilde{e}_x,\tilde{e}_z)$
is an even function of $\ti{e}_x$.\\

At the next order we get
\begin{equation}
\partial_{\tilde{t}} \tilde{h}^{t,0}+ \tilde{v}_x^{\#}\partial_{\tilde{x}}\tilde{h}^{t,0}
+\tilde{v}_x^{\#}\partial_y\tilde{h}^{t,1} =\frac{1}{\ti{\tau}}
\left( \ti{\Theta}[\ti{h}^0] \ti{l}\ti{M}-\ti{h}^{t,0}\right).
\label{o1}
\end{equation}
We consider now  separately bound and unbound molecules. \\

Let us first consider unbound molecules. 
Taking into account (\ref{po}) and multiplying (\ref{o1}) by 
$\tilde{v}_x^{\#}(\ti{x}/\varepsilon,\tilde{e}_x)^{-1}$
, we get, for unbound molecules
\begin{equation*}
\frac{1}{\tilde{v}_x^{\#}(\ti{x}/\varepsilon,\tilde{e}_x)}\partial_{\tilde{t}} \tilde{h}^{t,0}_u+  \partial_{\tilde{x}} 
\tilde{h}^{t,0}_u+ \partial_y\tilde{h}^{t,1}_u  = 
\frac{1}{\tilde{v}_x^{\#}(\ti{x}/\varepsilon,\tilde{e}_x)\ \ti{\tau}}
\left( \ti{{\Theta}}[\ti{h}^{t,0}]\ti{l}\ti{M}-\ti{h}^{t,0}_u\right).
\end{equation*}
Then averaging with respect to the fast variable over one period and taking into account that $\tilde{h}^{t,1}_u$ is periodic in $y$, we get
\begin{equation*}
\frac{1}{\tilde{w}_x(\ti{e}_x)}\partial_{\tilde{t}} \tilde{h}^{t,0}_u+ \partial_{\tilde{x}} 
\tilde{h}^{t,0}_u = 
\frac{1}{\tilde{w}_x(\tilde{e}_x)\ \ti{\tau}}
\left( \ti{\overline{\Theta}}[\ti{h}^{t,0}]\ti{l}\ti{M}-\ti{h}^{t,0}_u\right),
\end{equation*}
where
\begin{eqnarray*}
{\ti{w}_x(\ti{e}_x)} &=& {sgn(\ti{e}_x)} \left(\frac{1}{2}\int_{-1}^{+1} \tilde{v}_x^{\#}(y,\ti{e}_x)^{-1}dy\right)^{-1}, \\
\ti{\overline{\Theta}}[\ti{h}^{t,0}] &=&  \frac{1}{2\tilde{\overline{\sigma}}_x(\ti{e}_x)}\int_{-1}^{+1}\ti{\sigma}_x^{\#}(y,\ti{e}_x) 
\ti{\Theta}[\ti{h}^{t,0}](\ti{t},\ti{x},y,\ti{e}_z)\ dy. \label{tildovertheta}
\end{eqnarray*}
Finally, multiplying by $\tilde{w}_x(\ti{e}_x)$ we obtain
\begin{equation}
\partial_{\tilde{t}} \tilde{h}^{t,0}_u+\tilde{w}_x(\ti{e}_x) \partial_{\tilde{x}} \tilde{h}^{t,0}_u = 
\frac{1}{ \ti{\tau}}
\left( \ti{\overline{\Theta}}[\ti{h}^{t,0}]\ti{l}\ti{M}-\ti{h}^{t,0}_u\right), \label{closureu}
\end{equation}

Let us consider now bound molecules. Since $\tilde{h}^{t,0}_b(\tilde{t},\tilde{x},\tilde{e}_x,\tilde{e}_z)$
is an even function of $\ti{e}_x$ it is equal to its even part and thus the left-hand-side of ($\ref{o1}$) writes also for bound molecules
\begin{equation*}
\partial_{\tilde{t}} \tilde{h}^{t,0}_b+ {\frac{1}{2}}(\tilde{v}_x^{\#}(e_x)+\tilde{v}_x^{\#}(-e_x)) \partial_{\tilde{x}}\tilde{h}^{t,0}_b
+{\frac{1}{2}}\left( \tilde{v}_x^{\#}(e_x)\partial_y\tilde{h}^{t,1}(e_x)+\tilde{v}_x^{\#}(-e_x)\partial_y\tilde{h}^{t,1}(-e_x)\right). 
\label{o1-b}
\end{equation*}
But $\tilde{v}_x^{\#}(y,e_x)$ is an odd function of $e_x$, so that 
$(\tilde{v}_x^{\#}(e_x)+\tilde{v}_x^{\#}(-e_x)) \partial_{\tilde{x}}\tilde{h}^{t,0}_b$ cancels. Then multiplying by 
$\tilde{\sigma}_x^{\#}(\ti{x}/\varepsilon,\ti{e}_x)$ and averaging in the fast variable over one period (between 
$y_{-}(\ti{e}_x)$ and $y_{+}(\ti{e}_x)$), we get finally for the bound molecules
\begin{equation}
\partial_{\tilde{t}} \tilde{h}^{t,0}_b =
\frac{1}{\ti{\tau}} \left(  \ti{\overline{\Theta}}[\ti{h}^{t,0}]\ti{l}\ti{M}-\ti{h}^{t,0}_b\right),\label{closureb} 
\end{equation}
where
\begin{equation*}
\ti{\overline{\Theta}}[\ti{h}^{t,0}] = \frac{1}{(y_{+}(\ti{e}_x)-y_{-}(\ti{e}_x))\tilde{\overline{\sigma}}_x(\ti{e}_x)}\int_{y_{-}(\ti{e}_x)}^{y_{+}(\ti{e}_x)}
\ti{\sigma}_x^{\#}(y,\ti{e}_x) \ti{\Theta}[\ti{h}^{t,0}](\ti{t},\ti{x},y,\ti{e}_z)\ dy.
\end{equation*}

Rewriting equation (\ref{closureu}-\ref{closureb}) in dimension form and omitting the superscript $0$, we get
\begin{eqnarray}
\partial_{t} h^t_u(t,x,e_x,e_z)+ {w}_x(e_x)\partial_x h^t_u(t,x,e_x,e_z) &=& \frac{1}{{\tau}_{ms}} 
\left( \overline{\Theta}[{h}^t]l {M}-{h^t_u} \right),\label{closureu2}\\ 
\partial_{t} h^t_b(t,x,e_x,e_z) &=& \frac{1}{{\tau}_{ms}} 
\left( \overline{\Theta}[{h}^t]l {M}-{h^t_b} \right)\red{.}\label{closureb2}
\end{eqnarray}

\begin{proposition}
Under the assumptions 
(\ref{simplepot1}-\ref{simplepot2}-\ref{nocoll}-\ref{Ldeltaxstar}-\ref{Wminf0}-\ref{tauconst}-\ref{Uperiodic1}), 
the solutions of (\ref{eqht}) formally converge as
$\varepsilon =\delta/x^* \rightarrow 0$ to a distibution function $h^t(t,x,e_x,e_z)$, 
satisfying the following  "mesoscopic" kinetic model
\begin{eqnarray*}
\partial_{t}h^t + {w_x}(e_x)\partial_x h^t 
&=& \frac{1}{{\tau}_{ms}}
\left( \overline{\Theta}[{h}^t]l(e_z) {M}-{h}^t \right),\;\; \label{mesomod}
\end{eqnarray*}
where 
\begin{eqnarray*}
{{w}_x({e}_x)} &=& \chi_u(e_x)\ {sgn(e_x)}\left(\frac{1}{2}\int_{-1}^{+1} (e_x^2-2\hat{U}(y)/m)^{{-1/2}}dy\right)^{-1},
\label{averagev1}  \\             
\overline{\Theta}[{h}^t](e_x,{e}_z) &=& \frac{1}{(y_{+}(e_x)-y_{-}(e_x))\overline{\sigma}_x({e}_x)}\int_{y_{-}(e_x)}^{y_{+}(e_x)} {\sigma}_x^{\#}(y,{e}_x) 
\Theta[{h}^t](y,{e}_z)\ dy, \label{overlineTheta}\\
\Theta[{h}^t](y,e_z) &=& \frac{1}{{\tau}_z(e_z)}
\int_{z_{-}(e_z)}^{z_{+}(e_z)} (n[h/l](y,z)\sigma_z(z,e_z) /\gamma_1(y,z))dz, \\
\gamma_1(y,z) &=&  \frac{2 \pi kT}{m}  \exp \left(-\frac{\hat{U}(y)+W(z)}{kT}   \right),\\
n[g](y,z) &=& \int_{{\mathcal E}_z(z)} \int_{{\mathcal E}_x^{\#}(y)} g({e}_x,{e}_z) |{e}_x| {\sigma}_x^{\#}(y,{e}_x)
|{e}_z|{\sigma}_z(z,{e}_z)\  d{e}_x\ d{e}_z,\\
 \overline{\sigma}_x({e}_x) &=& \frac{1}{2} \int_{y_{-}(e_x)}^{y_{+}(e_x)} \left( \sqrt{e_x^2-\frac{2}{m}\hat{U}(y)}\right)^{-1}dy,\\
{\mathcal E}_x^{\#}(y) &=& \{e_x,\; e_x^2>2\hat{U}(y)/m \}. 
\end{eqnarray*}

\end{proposition}

\begin{remark}
\begin{enumerate}
\item The distribution functions $h^t_c$, for $c=b,u$ describe the number density of gas molecules
with velocity $(e_x,e_z)$ per x-unit, obtained by averaging the distribution functions $g^t$
over the periods of the tangential potential $U$. Then 
\[N(t,x)=\int_{y_{-}(e_x)}^{y_{-}(e_x)}\int_0^L n[h/l](t,x,y,z)\ dz\ dy = 2 
\int_{{e}_z} \int_{{e}_x} h |{e}_x| {\overline{\sigma}}_{x}({e}_x) d{e}_x\ d{e}_z \]
 is the number
density of molecules per unit-length, averaged over the periods of the tangential potential $U$. 
\item We can check that 
$\int_{{e}_z} \int_{{e}_x} (\overline{\Theta}[h]lM-h) |{e}_x| {\overline{\sigma}}_{x}({e}_x) d{e}_x\ d{e}_z =0$, 
which ensures the local conservation of mass.
\item In this model the unbound surface molecules move in the x-direction at an average velocity depending on their total  
tangential energy, or 
equivalently, on $e_x$. On the other hand the bound surface molecules are trapped in a well of the tangential potential $U$. 
Their average velocity is null and they relax inside a potential well toward the equilibrium (see equation (\ref{closureb2})). 
Let us notice that, contrary to (\cite{Kry2}), we do not need to assume that the bound molecules are at equilibrium.
\item We have assumed above that $L<< \delta<< x^*$ (\ref{Ldeltaxstar}). It could be interesting to drop this hypothesis 
and to consider the case where $L$ and $\delta$ are comparable. But, in such a case,  we cannot uncouple the asymptotic analysis 
of propositions 1 and 2 and we have to treat both of them together. 
\end{enumerate}
\end{remark}


\section{Diffusion models for surface molecules}

In section 2 we derived a multiphase model describing the coupling of the bulk flow with the
motion of molecules on the surface. This motion is given by a set of two coupled kinetic equations 
on variable $x$, including a Vlasov term due to the interaction potential parallel to the surface. 
  In this section we will consider integration
time much greater than the molecule-substrate relaxation time and thus we derive diffusion 
models. We consider the case of a smooth interaction potential parallel to the surface.
We proceed in several steps and we begin with simplifying assumptions. \\

In  a first
subsection we consider the same simple configuration as in section 3 where we assumed that the free molecules in the
surface layer can be neglected so that all surface molecules are trapped. As before it is reasonable to assume that 
\begin{equation} W_m=+\infty,  \label{Wminf} \end{equation}
so that when a molecule enters the surface layer, it cannot escape and moreover we neglect the
flux of incoming molecules. Thus there are no free 
molecules inside the surface layer and it is possible to describe the trapped surface molecules by the 
following closed model describing the evolution of the distribution function $g=g(v_x,e_z)$
\begin{equation}
\partial_t g +v_x\partial_x g - \frac{U'(x)}{m}\partial_{v_x}g= \frac{1}{\tau_{ms}}\left( \Theta [g]lM -g\right), \label{eqcin}
\end{equation}
 where
\begin{eqnarray*}
M(v_x,e_z)&=&\exp \left(-m(v_x^2+e_z^2)/2kT \right),
\end{eqnarray*}
and we denote
\begin{eqnarray*}
M_x(v_x)= \exp \left(-\frac{mv_x^2}{2kT} \right),\ M_z(e_z)=\exp \left(-\frac{me_z^2}{2kT} \right),\\
\gamma_x= \sqrt{\frac{2k\pi T}{m}},\
\gamma_z(z)= \int_{{\mathcal E}_z(z)}M_z(e_z)|e_z|\sigma_z(z,e_z)de_z= 
\sqrt{\frac{2k\pi T}{m}}\exp \left(-\frac{W(z)}{kT} \right).
\end{eqnarray*}
The diffusion limit is derived, first in the isothermal case and then extended to the non-isothermal
case. \\
Afterwards, in the second subsection we drop the assumption $W_m=+\infty$ and we extend the analysis of the diffusion
limit to the configuration of a flow in a narrow channel where the free molecules are not neglected but where there 
is no bulk flow. \\

Before deriving the diffusion models we need a technical lemma related to the collision operator ${\Theta}[g]lM-g$.
 Let us denote in the following,
\begin{eqnarray*}  
((. ))   &=& \int_{{e}_z} \int_{{v}_x}  d{v}_x\ d{e}_z.
\end{eqnarray*}
Then we have

\begin{lemma}
Let us consider the following equation for $g$:
\begin{equation}
\Theta[g](x,e_z)l(e_z)M(v_x,e_z) -g(x,v_x,e_z) + \psi(x,v_x)l(e_z)M_z(e_z) = 0,  \label{tetaeq}
\end{equation}
where $\psi$ is a given function.\\
 A necessary and sufficient solvability condition for this equation  is
\begin{equation}
\int_{v_x}\psi (x,v_x)\ dv_x = 0. \label{solvcond}
\end{equation}
Moreover if (\ref{solvcond}) is satisfied every solution of (\ref{tetaeq}) writes
\begin{equation*}
g=\beta(x)l(e_z) M(v_x,e_z)+ \psi(x,v_x)l(e_z)M_z(e_z), \label{soltetaeq}
\end{equation*}
where $\beta$ does not depend on $v_x$ and $e_z$ 
\end{lemma}
\begin{proof}  
(i) As we have noticed in remarks 1, $((\Theta[g]lM-g))=0 $
which gives the solvabiliy condition (\ref{solvcond}).\\
(ii) It can be easily checked that if $\beta= \beta (x)$ does not depend on 
$v_x$ and $e_z$, then $\Theta[\beta lM]=\beta$. Thus $g(x,v_x,e_z)= \beta(x)l(e_z)M(v_x,e_z)$ are solutions of 
equation (\ref{tetaeq}) with $\psi=0$. We can prove that they are the only solutions. Indeed if we look for a solution 
in the form $\beta(x,v_x,e_z)l(e_z)M(v_x,e_z)$, then $\beta(x,v_x,e_z)$ is solution of an integral equation
\begin{equation} 
\beta({x},{v}_x,{e}_z)=\int_{e'_z} \int_{v'_x} k({v}_x,{e}_z,{v'}_x,e'_z)\ 
\beta ({x},{v'}_x,e'_z)\  
d{v'}_{x}de'_{z},\label{intbeta}
\end{equation}
 where the kernel $k({v}_x,{e}_z,{v'}_x,e'_z)=k({e}_z,{v'}_x,{e'}_z)$ is defined by
\[k({e}_z,{v'}_x,{e'}_z)
=\int_{z_{-}(e_z,e'_z)}^{z_{+}(e_z,e'_z)} 
\frac{ {\sigma}_z({z},{e}_z) |{e'}_z|{\sigma}_z({z},e'_z)M_x({v'}_x)M_z(e'_z)}{\gamma_x\ \gamma_z({z})\ \tau_z(e_z)}\ d{z}, \]
where
\begin{eqnarray*}
 z_{-}(e_z,e'_z)=\max (z_{-}(e_z),z_{-}(e'_z)),\ z_{+}(e_z,e'_z)=\min(z_{+}(e_z),z_{+}(e'_z)), 
\end{eqnarray*} 
and this kernel $k$  satisfies 
\[\forall\  {e}_z,\;\; \int_{e'_z} \int_{v'_x}   k({e}_z,{v'}_x,{e'}_z)\ d{v'}_x\ d{e'}_z =1.\] 
Moreover, since the kernel $k$ is independent of $v_x$, so is $\beta$ and the integral equation (\ref{intbeta}) writes 
\begin{equation} 
\beta({x},{e}_z)=\int_{{e'}_z}  \hat{k}({e}_z,{e'}_z)\ 
\beta ({x},{e'}_z)\  d{e'}_{z},\label{intbeta2}
\end{equation}
where the kernel $\hat{k}({e}_z,{e'}_z)=\int_{{v'}_x}k({e}_z,{v'}_x,{e'}_z)d{v'}_x$ satisfies
\[\forall\  {e}_z,\;\; \hat{k}({e}_z,.) \in L^1_{{e'}_z}\;\; \mbox{ and}\;  
\int_{{e'}_z}   \hat{k}({e}_z,{e'}_z)\ d{e'}_z =1.\] 
It is easy to check that functions $\beta =\beta (x)$ (constant with respect to $e_z$) are solutions of this equation. Moreover
we prove now that they are the only solutions in $L^{\infty}_{e_z}$. First we prove by contradiction that the only solutions
of $\alpha(e_z)=\int_{e'_z} \hat{k}({e}_z,e'_z)\alpha(e'_z)d{e'}_z$ are constant functions.
If we denote $m= \mbox{ess sup}_{e'_z}\alpha(e'_z)$, since $\alpha$ is not a constant 
function of $e'_z$
\[   \exists \eta\ >0,\;\ \alpha(e'_z)\leq m-\eta, \; \mbox{ a.e. in $\; e'_z$ on a set}\;  {\mathcal F}_{x,\eta}\; 
\mbox{with positive measure in}\ e'_z,\]
so that
\begin{eqnarray*}
\alpha (e_z ) &=& \int_{e'_z} \hat{k}(e_z,e'_z)\alpha (e'_z)de'_z, \\
                      &\leq& m-\eta \int_{{\mathcal F}_{x,\eta}}\hat{k}(e_z,\epsilon_z)d\epsilon_z \leq m- \eta k_0,
\end{eqnarray*}
where
\[ 0<k_0= \sqrt{\frac{m}{2k\pi T}}\int_{{\mathcal F}_{x,\eta}}M_z(e'_z)de'_z\leq \int_{} \frac{\sigma_z(z,e_z)}{\tau_z(e_z)\gamma_z(z)}
\int_{{\mathcal F}_{x,\eta}}|e'_z|\sigma_z(z,e'_z)M_z(e'_z)de'_z\ dz.\]
This last inequality is in contradiction 
with the definition of $m$ so that we conclude that necessarily $\alpha$ does not depends on $e_z$. We easily deduce of 
this property that the only solutions of (\ref{intbeta2}) are functions $\beta (x,.)$ which, for almost $x$, are constant
functions in $e_z$. Indeed, for $\xi$ given in $C^0_{K}$, let us denote $\alpha_{\xi}(e_z)=\int_x \xi (x)\beta(x,e_z)dx$. Then  
$\alpha_{\xi}$ is solution of the
integral equation $\alpha_{\xi}(e_z)=\int_{e'_z} \hat{k}({e}_z,e'_z)\alpha_{\xi}(e'_z)d{e'}_z$, so that $\alpha_{\xi}$ is a constant. But
$\xi  \rightarrow \alpha_{\xi}$ is a linear form that can be written $\int_x\zeta (x)\xi(x)dx$, so that almost everywhere in $x$
we have $\beta(x,e_z)=\zeta (x)$. Thus every solution of the homogeneous equation (\ref{tetaeq}) writes $\beta (x) l(e_z)M(v_x,e_z).$\\

(iii) Let us now consider the nonhomogeneous equation. First, we remark that (\ref{solvcond}) implies that
\begin{equation*}
n[\psi M_z]=\left(\int_{v_x}\psi(x,v_x)dv_x\right)\left(\int_{e_z}|e_z|\sigma_z(z,e_z)M_z(e_z)de_z\right)=0,
\end{equation*} 
so that $\psi l M$ is a particular solution of (\ref{tetaeq}), and then, from (ii), every solution writes
\[g(x,v_x,e_z)=\beta(x)l(e_z)M(v_x,e_z)+\psi(x,v_x)l(e_z)M_z(e_z).\]
\end{proof}

\subsection{Diffusion model for the trapped surface-molecules}
For the sake of simplicity, we assume in a first step that the temperature $T$ is a
fixed given constant.   

\subsubsection{The isothermal case}
Under the above assumptions the  kinetic model obtained in the previous sections writes
\begin{equation}
\partial_t g +v_x\partial_x g - \frac{U'(x)}{m}\partial_{v_x}g= \frac{1}{\tau_{ms}}\left( \Theta [g]lM -g\right). \label{eqcinbis}
\end{equation}
{To obtain a dimensionless form of this equation we introduce the following reference quantities. ${z}^*=L$ is a 
reference length in the z-direction, ${{v}^*}=\sqrt{2kT/m} $ is a reference velocity, ${t}_d^*$ is a reference 
{\em diffusion time},  ${\tau}^*_{ms}$ is a reference {\em molecule-phonon relaxation time}, $\tau_z^*=z^*/v^*$ is 
a reference time for crossing the surface layer, ${n}^*$ is a reference number density and ${g}^*=({n}^*z^*)/{v^*}^2$, 
$U^*=m{v^*}²$ 
and $\tilde{M}(\ti{v}_x,\ti{e}_z) = 
\exp \left(-(\ti{v}_x^2+\ti{e}_z^2)\right)$. Moreover we introduce another reference time $t_c^*$ such that}
\begin{equation*}
     {\tau}^*_{z}\;  <<\; {\tau}^*_{ms}\;  <<\; {t}_c^*\; <<\; {t}_d^*,  \label{hypdiff}
\end{equation*}
and the small parameter $\varepsilon$, and we set
\begin{equation*}
     \varepsilon =\frac{{\tau}^*_{ms}}{{t}_c^*} =\frac{{t}_c^*}{{t}_d^*}. \label{epsi}
\end{equation*}
{Finally  we introduce ${{x}^{**}}$, a reference length in the x-direction defined by}
\begin{equation*}
x^{**}=v^* t_c^*.
\end{equation*}
We rescale the problem according to 
\begin{equation}
 \tilde{x}=\frac{x}{{{x}^{**}}}, 
\tilde{z}=\frac{z}{{z}^*},
\tilde{l}=\frac{l}{{z}^*},
\tilde{t}=\frac{t}{{t}_d^*},
\tilde{\tau}_{ms}=\frac{{\tau}_{ms}}{{\tau}^*_{ms}},
 \tilde{{v}}_x=\frac{{v}_x}{{{v}^*}}, 
\tilde{e}_{x/z}=\frac{e_{x/z}}{v^*},
\label{scaling-diff}
\end{equation}
and
\begin{equation}
\tilde{g}=\frac{g{v^*}^2}{n^*z^*},\;\;\;\; \tilde{U}=\frac{U}{U^*}, \;\;\;\;
\tilde{{\Theta}}= \frac{{\Theta}{v^*}^2}{n^*}.\label{scaling-diff-2}
\end{equation}
With the previous notations and the above assumption, the dimensionless form of equation (\ref{eqcinbis}) writes
\begin{eqnarray}
\varepsilon \partial_{\tilde{t}}\ti{g}+
\tilde{{v}}_x
\partial_{\tilde{x}}{\ti{g}}- \ti{U}'(\ti{x})\partial_{\ti{v}_x}\ti{g}&=&
\frac{1}{\varepsilon \tilde{\tau}_{ms}} 
\left(\ti{{\Theta}}[\ti{g}] \ti{l}\ti{M}-
{\tilde{g}} \right)  \label{adimavkin}, 
\end{eqnarray}
We consider the asymptotic analysis of this problem when $\varepsilon$ tends to
$0$. We look for
\begin{eqnarray} 
\tilde{g} &= &\tilde{g}^{(0)}+\varepsilon \tilde{g}^{(1)}+ \varepsilon^2 \tilde{g}^{(2)} + .... \label{data}.
\end{eqnarray}
       
Inserting (\ref{data}) in (\ref{adimavkin}) and balancing order 
by order in $\varepsilon$ we get\\
\noindent {\em At leading order}

\begin{equation}  
\tilde{g}^{(0)} = \ti{{\Theta}}[\tilde{g}^{(0)}] \ti{l} \tilde{M}. \label{orderzdiff}  
\end{equation}
From the dimensionless form of Lemma 1 the solutions of (\ref{orderzdiff}) write
\begin{equation}
\ti{g}^{(0)} = \alpha (\ti{t},\ti{x})\ti{l}\ti{M}= \frac{\ti{N}^0(\ti{t},\ti{x})}{\ti{\gamma}}\ti{l}\ti{M},
\label{gzero}
\end{equation}
where $\ti{N}^0(\ti{t},\ti{x})$ is defined by
\begin{equation*}
   \ti{N}^0(\ti{t},\ti{x})= \alpha (\ti{t},\ti{x})\ \ti{\gamma},
\end{equation*}
and  where the constant $\ti{\gamma}$ (depending only on $W$ and $T$) is defined by
\begin{equation*}
\ti{\gamma}=  \int_{{\ti{e}_z}} \int_{{\ti{v}_x}} \ti{M}({\ti{v}_x},{\ti{e}_z})
\ti{l}({\ti{e}_z})\ d{\ti{v}_x}d{\ti{e}_z}. 
\end{equation*}

 Then we get \\
\noindent {\em at order +1}
\begin{eqnarray*}
 \tilde{{v}_x}\ \partial_{\tilde{x}}\tilde{g}^{(0)}-\ti{U}'(\ti{x})\partial{\ti{v}_x}\ti{g}^{(0)} &=& 
\frac{1}{\tilde{\tau}_{ms}}\left(\ti{{\Theta}}[\ti{g}^{(1)}]\ti{l}\ti{M} -\tilde{g}^{(1)} \right).  
\label{order1}
\end{eqnarray*}
We notice that the left-hand-side of this relation writes 
\[\left( \tilde{{v}_x}\ \partial_{\tilde{x}}\ti{N}\frac{M_x(v_x)}{\ti{\gamma}_x}+2\ti{U}'(\ti{x})\ti{v}_x\ti{N}\ti{M}_x(\ti{v}_x)\right)
\ti{l}(\ti{e}_z)\ti{M}_z(\ti{e}_z), \]
so that we can use the dimensionless form of Lemma 1 and thus solutions of this relation write
\begin{eqnarray}
\tilde{g}^{(1)} &=& \ti{\alpha}^1(\ti{t},\ti{x})\ti{l}\ti{M} 
- \frac{\ti{\tau}_{ms}}{\ti{\gamma}}  \left(\partial_{\ti{x}}\ti{N}^0(\ti{t},\ti{x})\right) \ti{l}\ti{{v}_x}\ti{M}
+ \frac{\ti{\tau_{ms}}}{\ti{\gamma}}\ti{U}'(\ti{x})(\partial_{\ti{x}}\ti{N}^0(\ti{t},\ti{x}))\ti{l}\ti{M}
\label{order1diff}.\\
\tilde{g}^{(1)} &=& \ti{\alpha}^1(\ti{t},\ti{x})\ti{l}\ti{M} 
- \frac{\ti{\tau}_{ms}}{\ti{\gamma}}  \left(\partial_{\ti{x}}\ti{N}^0(\ti{t},\ti{x})\right)(\ti{{v}_x}-\ti{U}'(\ti{x}))\ti{l}\ti{M}
\label{order1diff}.
\end{eqnarray}

\noindent {\em at order +2}

\begin{eqnarray*}
\partial_{\ti{t}} \ti{g}^{(0)}  +\ \tilde{{v}}_x\ \partial_{\tilde{x}}\tilde{g}^{(1)} 
-\ti{U}'(\ti{x})\partial{\ti{v}_x}\ti{g}^{(1)}&=& 
\frac{1}{\tilde{\tau}_{ms}}\left(\ti{{\Theta}}[\ti{g}^{(2)}]\ti{l}\ti{M}- \tilde{g}^{(2)} \right).
\label{order2diff}
\end{eqnarray*}
As above and taking into account (\ref{order1diff}), the left hand side of this relation can be written in the form 
$\ti{\psi}(\ti{x},\ti{v}_x)\ti{l}(\ti{e}_z)\ti{M}_z(\ti{e}_z)$, and thus existence of a solution in $\tilde{g}^{(2)}$ is
ensured under the solvability condition
\begin{eqnarray}
\int_{\tilde{{v_x}}} \int_{\ti{e}_z} \left(\partial_{\ti{t}} \ti{g}^{(0)}  
+\ \tilde{{v_x}}\ \partial_{\tilde{x}}\tilde{g}^{(1)} 
-\ti{U}'(\ti{x})\partial{\ti{v}_x}\ti{g}^{(1)} \right)d\ti{e}_z d\tilde{{v_x}}&=& 0.
\label{order2diffint}
\end{eqnarray}
Inserting (\ref{gzero}-\ref{order1diff}) into (\ref{order2diffint}) and noting that 
$ \int_{\tilde{{v_x}}} \int_{\ti{e}_z}\ti{U}'(\ti{x})\partial{\ti{v}_x}\ti{g}^{(1)} d\ti{e}_z d\tilde{{v_x}}=0$,
we get
\begin{equation}
\partial_{\ti{t}} \ti{N}^0 - {\ti{D}_0^{(n)} \partial^2 _{\ti{x}^2} \ti{N}^0} -\ti{\tau}_{ms}\partial_{\ti{x}}\left( \ti{U}'(\ti{x})\ti{N}^0\right) =0, \label{adimdiff}
\end{equation}
where  the diffusion coefficient $\ti{D}_0^{(n)}$ is given by
\begin{equation*}
\ti{D}_0^{(n)}
=  \frac{ \tilde{\tau}_{ms}}{\ti{\gamma}} 
\int_{{\ti{e}_z}} \int_{{\ti{v}_x}}  \tilde{{v}_x}^2\  \ti{M}({\ti{v}_x},{\ti{e}_z})
\ti{l}({\ti{e}_z})\ d{\ti{v}_x}d{\ti{e}_z}.
\end{equation*}
Then we come back to dimension quantities and we denote $N(t,x)= n^* z^* \ti{N}^0(\ti{t},\ti{x})$. With this definition 
${g^0}=(N/\gamma)lM$, where $\gamma   =  \langle \langle  {M} \rangle \rangle$, and 
\[   \langle\langle . \rangle\rangle=\int_{{e}_z} \int_{{e}_x} 
\cdot\  d{v}_x\ d{e}_z, \]
so that
the number density of molecules 
\[n[g^0/l]= \int_{{\mathcal E}_z(z)}  \int_{v_x} \frac{g^0(t,x,v_x,e_z)}{l(e_z)}|e_z| 
\sigma_z(z,e_z)\ dv_xde_z \]
satisfies
\[ \int_{0}^{L} n[g/l]\ dz  = \frac{N(t,x)}{\gamma}\int_{e_z}  \int_{v_x} 
M(v_x,e_z)dv_xde_z = N(t,x). \]
Thus $N(t,x)$ can be interpreted as the number density of molecules per unit x-length in the surface layer, obtained by
integrating (with respect to $z$) over the width of the layer and from (\ref{adimdiff}) we get
\begin{equation*}
\partial_t {N} - {D_0^{(n)}\partial^2_{{x}^2} {N}}-\tau_{ms}\partial_{x}\left(\frac{U'(x)}{m} N\right) =0,
\end{equation*}
Finally we have obtained the following result

\begin{proposition}
Under the hypothesis (\ref{simplepot1}-\ref{simplepot2}-\ref{nocoll}-\ref{tauconst}-\ref{Wminf} ), 
the solutions of (\ref{eqcin}) formally converge as 
$\varepsilon(={t}_c^*/{t}_d^*={\tau}_{ms}^*/t_c^*) \rightarrow 0$
to $\left( N(t,x)/\gamma \right)l(e_z) M(v_x,e_z)$ where the function $N(t,x)$ is solution of the following diffusion 
equation
\begin{equation*}
\partial_tN-{D_0^{(n)}\partial^2_{x^2}N}-\tau_{ms}\partial_{x}\left(\frac{U'(x)}{m} N\right) =0,
\end{equation*} 
where
\begin{equation*}
D_0^{(n)}= \frac{\tau_{ms}}{\gamma}  \langle\langle  {{v}_x}^2\ l{M} \rangle\rangle, \; 
\gamma = \langle\langle M \rangle\rangle, \; {M}(v_x,e_z)= e^{-m(v_x^2+e_z^2)/2kT}. \label{diffcoeff}
\end{equation*}
\end{proposition}

\begin{remark} 

\begin{enumerate}
\item This diffusion equation describes the gas flow in the limit $\varepsilon \rightarrow 0$ only on a
long time scale $\approx t_d$, or, which is in some sense equivalent, for an initial distribution of the form 
$(N/\gamma)M(v_x,e_z)$. 
It does not give a correct description of the flow for a general initial condition and on a short time scale. 
\end{enumerate} 
\end{remark}

\subsubsection{The non-isothermal case}

In this section we assume that the temperature is a (given) function of $x$, and we perform the same
asymptotic analysis as in the previous section. We only indicate the modification induced by this
assumption.  We first notice that  $M =  \exp \left( -(v_x^2+e_z^2)/(2kT(x)\right)$ and that 
$\gamma= \gamma (T(x))$. Moreover 
the derivative of $\gamma$ with respect to the temperature, denoted by $\gamma '$, is given by
\[  \gamma '(T) =  \partial_{ T}  \gamma = \frac{1}{2k T^2}\langle \langle (v_x^2+e_z^2) M  \rangle \rangle.   \]

The asymptotic analysis follows the same steps as in the previous section. The first difference 
occurs at the identification of the expansion at order +1. 
Taking into account the variation of $T$ with respect to $x$, the first term in the expansion
reads now  instead of (\ref{order1diff})
\begin{eqnarray*}
\tilde{g}^{(1)}& =& \ti{N}^1\ti{l}\ti{M} - \frac{\ti{\tau}_{ms}\ti{{v}_x}}{\ti{\gamma}} 
\left(  \partial_{\tilde{x}} \ti{N}^0
- \ti{N}^0 \partial_{\ti{x}} T(x) \left( \frac{\ti{\gamma} '(T(x))\ }{\ti{\gamma}}
+ \frac{\ti{v}_x^2+\ti{e}_z^2 }{ 2k T(x)^2} \right)\right)\ti{M}\nonumber \\ 
&&+\ti{\tau}_{ms}\ti{U}'(\ti{x})\partial_{\ti{v}_x}\ti{g}^0,
 \label{order1diffT} 
\end{eqnarray*}
where the effect of the temperature gradient is included. Finally the asymptotic analysis leads
 to the diffusion equation
\begin{equation}
\partial_tN-\partial_x(D_0^{(n)}\partial_x {N})- \partial_x(D_0^{(T)}\partial_x T)-\tau_{ms}\partial_x \left(\frac{U'(x)}{m}N \right) =0,
\label{difeq-nT}
\end{equation}
where
\begin{eqnarray*}
D_0^{(n)}&=&D_0^{(n)}(T)= \frac{\tau_{ms}}{\gamma (T(x))} \langle\langle  {v}_x^2 lM\rangle\rangle,\\
D_0^{(T)} &=& D_0^{(T)}(N,T)=
\frac{\tau_{ms}}{\gamma (T(x))}  \langle \langle  {v}_x^2\left( \frac{m(v_x^2+e_z^2)}{2kT(x)^2}-
\frac{ \gamma '(T(x))}{\gamma^2 (T(x))} \right)l M\rangle\rangle\ N.
\end{eqnarray*}

Of course, this approach can be easily extended to the case where the temperature field is not given 
but determined by a conduction equation (for instance). \\

It is sometimes preferred to write the diffusion equation (\ref{difeq-nT}) with the mass flux defined with 
respect to $\partial_xp$ rather than with respect to $\partial_xN$. 
Assuming that the pressure of the gas
is given by the ideal gas law $p= kNT$, the diffusion equation (\ref{difeq-nT}) writes
 \begin{equation*}
\partial_tN-\partial_x(C_0^{(p)}\partial_x p)- \partial_x(C_0^{(T)}\partial_x T) 
=0, \label{difeq-pT}
\end{equation*}
where
\begin{eqnarray*}
C_0^{(p)}&=&\frac{1}{kT}D_0^{(n)}\\
C_0^{(T)} &=& D_0^{(T)}-\frac{N}{T}D_0^{(n)}.
\end{eqnarray*}

\subsection{Diffusion model for trapped and free molecules in a narrow channel } 

In the previous subsection we derived a diffusion limit of the tangential model under the
assumption that all molecules in the surface layer are trapped. Now we drop this limitation so that the
flow of trapped molecules is coupled with the flow outside the surface layer by the free molecules. We 
consider the following configuration: the gas molecules move in a narrow channel with diameter $2L$, with
a lower boundary (denoted $1$), located at $z=0$ and an upper boundary (denoted $2$), located at $z=2L$ 
(see figure \ref{narrowchannel}).
In this simple configuration there is no bulk flow and the incoming free molecules in the lower surface 
layer are the free molecules going out of the upper surface layer and conversely. We assume moreover that
the two boundaries are similar (same temperature, same interaction potential) so that
\begin{equation}
T_1=T_2, \;\;\;\;\;U_1(x)=U_2(x)=U(x), \;\;\;\;\;W_1(z)=W_2(2L-z). \label{simbound}
\end{equation}
\begin{figure}
\begin{center}
\includegraphics[height= 3. cm, angle=0]{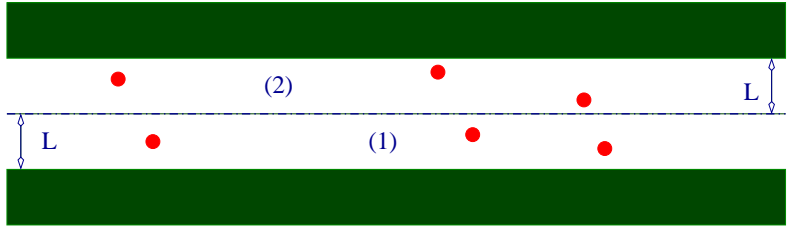}\\
\vspace*{0.3 cm}
\caption{ narrow channel with two surface layers. \label{narrowchannel} }
\end{center}
\end{figure}
Using the same approach as
in section 2, the flow of molecules in the channel can be described (in the isothermal case) by the following 
system of two coupled monodimensional kinetic equations
\begin{eqnarray}
\partial_t g_1+ {v_x}\partial_x g_1-\frac{U'(x)}{m}\partial_{v_x}g_1 = 
Q_{ph}[g_1]+\frac{\chi^f(e_z)}{2{\tau}_z(e_z)} \left(g_2(-|e_z|)-g_1(|e_z|)\right)
, \label{eqh-nc-1} \\
\partial_t g_2+ {v_x}\partial_x g_2-\frac{U'(x)}{m}\partial_{v_x}g_2 = 
Q_{ph}[g_1]+\frac{\chi^f(e_z)}{2{\tau}_z(e_z)} \left(g_1(|e_z|)-g_2(-|e_z|)\right)
, \label{eqh-nc-2} 
\end{eqnarray}
where $g_1$ is the distribution function describing the (trapped and free) molecules inside the lower surface layer and 
$g_2$ is the distribution function describing the (trapped and free) molecules inside the upper surface layer. From hypothesis
(\ref{simbound}) the times $\tau_z$ is the same in the two equations and we have the same property for 
$\tau_{ms}$, for the length $l$, for the characteristic function $\chi^f$, for the operator ${\Theta}$ and the 
distribution $M$.

We denote $g_*=g_1+g_2$,
the distribution function of molecules in the channel. Then if we add equation (\ref{eqh-nc-1}) and equation (\ref{eqh-nc-2}), the sum
of the coupling terms vanishes since the free molecules outgoing from the lower surface layer are the free molecules incoming into 
the upper surface layer and conversely. Thus (using that the operator ${\Theta}$ is linear in $g$),  we obtain for $g_*$ 
the following closed kinetic equation
\begin{eqnarray*}
\partial_t g_*+ {v_x}\partial_x g_* -\frac{U'(x)}{m}\partial_{v_x}g_*&=& 
Q_{ph}[g_*]=\frac{1}{\tau_{ms}}\left({\Theta}[g_*]lM-g_*\right), \label{eqh-nc-*}
\end{eqnarray*}
which is very similar to (\ref{eqcinbis}), but includes the free molecules ($g_i=\chi^tg_i^t+\chi^fg_i^f,\ i=1,2$).
If we rescale this equation as in
(\ref{scaling-diff}-\ref{scaling-diff-2}), we get the following dimensionless equation for $g_*$
\begin{eqnarray*}
\varepsilon \partial_{\ti{t}} \ti{g}_*+ \ti{v}_{{x}}\partial_{\ti{x}} \ti{g}_* -\ti{U}'(\ti{x})\partial_{\ti{v}_x}\ti{g}_*&=& 
\frac{1}{\varepsilon\ \ti{\tau}_{ms}}\left(\ti{{\Theta}}[\ti{g}_*]\ti{l}\ti{M}-\ti{g}_*\right), \label{eqh-nc-*-adim}
\end{eqnarray*}
and the same asymptotic analysis as in section 4.1 leads to the following diffusion equation
for $N_*= \int_0^Ln[(g_1+g_2)/l]dz$
\begin{equation}
\partial_tN_*-{D_0^{(n)}\partial^2_{x^2}N_*}-\tau_{ms}\partial_x\left(\frac{U'(x)}{m}N^* \right) =0, \label{difflimit*}
\end{equation} 
where
\begin{equation*}
D_0^{(n)}= \frac{\tau_{ms}}{\gamma}  \langle\langle  {{v}_x}^2 l{M} \rangle\rangle, \; 
\gamma = \langle\langle M \rangle\rangle, \; {M}(v_x,e_z)= e^{-m(v_x^2+e_z^2)/2kT}. 
\end{equation*}

Let us look now for a diffusion limit of equations (\ref{eqh-nc-1}-\ref{eqh-nc-2}). We need some more assumptions. 
We rescale these equations as in (\ref{scaling-diff}-\ref{scaling-diff-2}) with, in addition  
\begin{equation*}
 \ti{\tau}_z=\frac{{\tau}_z}{\tau^*_z}, \;\;\;\;\; \ti{\chi}^f\ti{g}_i= \frac{\chi^fg_i}{(n_f^*/{v^*}^2)},
\end{equation*}
where ${\tau^*_z=z^*/v^*}$. We denote
\begin{equation*}
\varepsilon_0=\frac{{\tau}_z^*}{t_c^*}=\frac{z^*}{x^*}, \label{eps0}
\end{equation*}
 and as in 4.1.1 we assume that
\begin{equation*}
 \varepsilon=\frac{{\tau}_{ms}^*}{t_c^*} =\frac{tc^*}{t_d^*}.   \label{eps}
\end{equation*}
With these notations (\ref{eqh-nc-1}) writes in dimensionless form
\begin{eqnarray*}
\varepsilon \partial_{\ti{t}} \ti{g}_1+ \ti{v}_{{x}}\partial_{\ti{x}} \ti{g}_1 
-  \ti{U}'(\ti{x})\partial_{\ti{v}_x} \ti{g}_1&=&
\frac{1}{\varepsilon}\left(\ti{{\Theta}}[\ti{g}_1]\ti{l}\ti{M}-\ti{g}_1\right)   \\
&&+ \frac{n_f^* t_c^*}{n^*\tau_z^*}
\frac{ \chi^f(\ti{e}_z)}{ 2\ti{{\tau}}_z(\ti{e}_z)} \left(\ti{g}_2(-|\ti{e}_z|)-\ti{g}_1(|\ti{e}_z|)\right),
\end{eqnarray*}
and we have a similar formula for (\ref{eqh-nc-2}). The right-hand-side of those two equations contains a collision term (of order $\frac{1}{\varepsilon}$)
and a coupling term which makes the two equations relax one towards the other. To go further we must precise the relative size of the coupling term 
with respect to the collision term which is determined by the ratio $\frac{n_f^* t_c^*}{n^*\tau_z^*}=\frac{n_f^*}{n^*\varepsilon_0}$. 
Different regimes can be encountered according to the size of $n_f^*/n^*$ compared with the small quantities $\varepsilon$ and $\varepsilon_0$,
 and we study several ones in the following. \\

\subsubsection{Strong coupling of the two surface layers}
	We assume there that
\begin{equation*}
\varepsilon_0 \leq \varepsilon ,\;\;     \frac{n_f^*}{n^*}=\frac{\varepsilon_0}{\varepsilon} .  \label{strongcoupling}
\end{equation*}
 The first assumption means that $\tau_z^* \leq \tau_{ms}^*$. The second assumption means that the ratio of free molecules is of order 
$\frac{\varepsilon_0}{\varepsilon}$, so that the number of free molecules can be (smaller than but) comparable to the total number of molecules
(when $\varepsilon_0= \varepsilon$). This is a reasonable assumption when $W_m \approx kT$. Under those assumptions the ratio  
$\frac{n_f^* t_c^*}{n^*\tau_z^*}=1/\varepsilon$ so that the coupling term of the two equations is strong and is comparable to the collision term.
 Then the dimensionless form of the system is given by
\begin{eqnarray}
\varepsilon \partial_{\ti{t}} \ti{g}_1+ \ti{v}_{{x}}\partial_{\ti{x}} \ti{g}_1 -\ti{U}'(\ti{x})\partial_{\ti{v}_x}\ti{g}_1&=& 
\Lambda[\ti{g}_1,\ti{g}_2] , \label{eqh-nc-1-adim}\\
\varepsilon \partial_{\ti{t}} \ti{g}_2+ \ti{v}_{{x}}\partial_{\ti{x}} \ti{g}_2 -\ti{U}'(\ti{x})\partial_{\ti{v}_x}\ti{g}_2&=& 
\Lambda[\ti{g}_2,\ti{g}_1] , \label{eqh-nc-2-adim}
\end{eqnarray}
where $\Lambda[\ti{g}_a,\ti{g}_b]=\frac{1}{\varepsilon\ \ti{\tau}_{ms}}\left(\ti{{\Theta}}[\ti{g}_a]\ti{l}\ti{M}-\ti{g}_a\right)+ 
\frac{ \chi^f}{ 2\varepsilon\ \ti{{\tau}}_z} \left(\ti{g}_b(-|\ti{e}_z|)-\ti{g}_a(|\ti{e}_z|)\right)$.\\
Since we have derived a diffusion limit for $g_*=g_1+g_2$, in order to get the diffusion limit of the system of  
equations modeling the evolution of $g_1$ and $g_2$ it is sufficient to study the diffusion limit of the equation 
giving the evolution of $g_{**}=g_1-g_2$ obtained by subtracting (\ref{eqh-nc-2-adim}) from (\ref{eqh-nc-1-adim}). 
Looking for $\ti{g}_{**}=\ti{g}_{**}^0+\varepsilon\ti{g}_{**}^1+...$, inserting in (\ref{eqh-nc-2-adim})-(\ref{eqh-nc-1-adim})
and balancing order by order we obtain\\
\noindent
{\em At the leading order}
\begin{equation*}
 \left(\ti{{\Theta}}[\ti{g}_{**}^0]\ti{l}\ti{M}-\ti{g}_{**}^0\right)-\frac{\ti{\tau}_{ms} \chi^f(\ti{e}_z)}{ 2\ \ti{{\tau}}_z(\ti{e}_z)} 
(\ti{g}_{**}^0(|\ti{e}_z|)+\ti{g}_{**}^0(-|\ti{e}_z|)) = 0.
\label{order0strong}
\end{equation*}
We first remark that this relation implies that $\ti{g}_{**}^0$ is an even function of $\ti{e}_z$, so that it can be written
\begin{equation}
\left(\ti{{\Theta}}[\ti{g}_{**}^0]\ti{l}\ti{M}-\ti{g}_{**}^0\right)-\frac{\ti{\tau}_{ms} \chi^f(\ti{e}_z)}{ \ti{{\tau}}_z(\ti{e}_z)}\ti{g}_{**}^0=0.
\label{order0strongbis}
\end{equation}
 To solve this equation we proceed as in the proof of Lemma 1 and we look for $\ti{g}_{**}=\ti{\beta}(\ti{x},\ti{v}_x,\ti{e}_z)\ti{l}\ti{M}$.
Inserting in (\ref{order0strongbis}) we find for $\ti{\beta}$ the following integral equation
\begin{equation}
\left(1+\frac{ \ti{\tau}_{ms}   \chi^f(\ti{e}_z)}{\ti{{\tau}}_z(\ti{e}_z)}\right)\ti{\beta}(\ti{x},\ti{v}_x,\ti{e}_z)=
\int_{\ti{\epsilon}_z}\int_{\ti{w}_x}\ti{k}(\ti{v}_x,\ti{e}_z,\ti{w}_x,\ti{\epsilon}_z)\ti{\beta}(\ti{x},\ti{w}_x,\ti{\epsilon}_z)
d\ti{w}_xd\ti{\epsilon}_z, \label{inteqstrong}
\end{equation}
where the kernel $\ti{k}$ ( the dimensionless form of the kernel $k$ introduced in Lemma 1) is nonnegative and satisfies
\begin{equation}
\int_{\ti{\epsilon}_z}\int_{\ti{w}_x}\ti{k}(\ti{v}_x,\ti{e}_z,\ti{w}_x,\ti{\epsilon}_z) d\ti{w}_xd\ti{\epsilon}_z=1. \label{propkstrong}
\end{equation}
Thus from (\ref{inteqstrong}), we deduce that $\ti{\beta}$ vanishes or cannot be independent of $\ti{v}_x$ and $\ti{e}_z$ and moreover
\begin{equation*}
\min_{\ti{w}_x,\ti{\epsilon}_z}\ti{\beta}(\ti{x},\ti{w}_x,\ti{\epsilon}_z)
 \leq \left(1+\frac{ \ti{\tau}_{ms}\chi^f(\ti{e}_z)}{\ti{{\tau}}_z(\ti{e}_z)}\right)\ti{\beta}(\ti{x},\ti{v}_x,\ti{e}_z) \leq
\max_{\ti{w}_x,\ti{\epsilon}_z}\ti{\beta}(\ti{x},\ti{w}_x,\ti{\epsilon}_z).
\end{equation*}
Since $\left(1+\frac{ \ti{\tau}_{ms}\chi^f(\ti{e}_z)}{\ti{{\tau}}_z(\ti{e}_z)}\right)>1$ for $|\ti{e}_z|>\ti{W}_m$, 
this relation implies that $\ti{\beta}$ reaches its maximum  for a value $(\ti{w}_x^*,\ti{e}_z^*)$ such that 
$|\ti{e}_z^*|\leq\ti{W}_m$.
Then we have
\begin{equation*}
\ti{\beta}(\ti{x},\ti{v}_x^*,\ti{e}_z^*)=
\int_{\ti{\epsilon}_z}\int_{\ti{w}_x}\ti{k}(\ti{v}_x^*,\ti{e}_z^*,\ti{w}_x,\ti{\epsilon}_z)\ti{\beta}(\ti{x},\ti{w}_x,\ti{\epsilon}_z)
d\ti{w}_xd\ti{\epsilon}_z.
\end{equation*}
But this relation and (\ref{propkstrong}) imply that  $\ti{\beta}$ is independent of $\ti{w}_x$ and $\ti{e}_z$, and 
thus, from the remark above, that $\ti{\beta}=0$. Consequently we have proved that $\ti{g}_{**}^0=0$. \\
Using this result we check that $\ti{g}_{**}^1$ satisfies the same equation as  $\ti{g}_{**}^0$ so that we can also conclude that  
$\ti{g}_{**}^1=0$ and so on. Finally, since $\ti{g}_{**}^0=0$, we conclude that $N_1=N_2=N_*/2$ and that $N_1$ and $N_2$ satisfy the
same equation as $N_*$ (\ref{difflimit*}). In fact since the coupling is strong, the distribution functions $g_1$ and $g_2$ relax 
very fast to the same limit $N_*/(2\gamma)lM$, in a time shorter than the diffusion time.

\subsubsection{Moderate coupling of the two surface layers}

We assume now that
\begin{equation*}
\frac{n_f^*}{n^*}=\varepsilon_0,\;\;   \label{moderatecoupling}
\end{equation*}
so that the ratio $\frac{n_f^* t_c^*}{n^*\tau_z^*}=1$. This assumption means that the number of free molecules is 
much smaller than the total number of molecules, which is reasonable when $W_m>>kT$. Under this assumptions the coupling term
is moderate (of order $1$) and smaller the collision term (of order $1/\varepsilon$).
\\Then the dimensionless form of the system is
given by
\begin{eqnarray*}
\varepsilon \partial_{\ti{t}} \ti{g}_1+ \ti{v}_{{x}}\partial_{\ti{x}} \ti{g}_1 -\ti{U}'(\ti{x})\partial_{\ti{v}_x}\ti{g}_1&=& 
\Lambda[\ti{g}_1,\ti{g}_2] , \label{eqh-nc-1-adim-mod}\\
\varepsilon \partial_{\ti{t}} \ti{g}_2+ \ti{v}_{{x}}\partial_{\ti{x}} \ti{g}_2 -\ti{U}'(\ti{x})\partial_{\ti{v}_x}\ti{g}_2&=& 
\Lambda[\ti{g}_2,\ti{g}_1] , \label{eqh-nc-2-adim-mod}
\end{eqnarray*}
where \[\Lambda[\ti{g}_a,\ti{g}_b] = \frac{1}{\varepsilon\ \ti{\tau}_{ms}}\left(\ti{{\Theta}}[\ti{g}_a]\ti{l}\ti{M}-\ti{g}_a\right)+ 
\frac{ \chi^f}{ 2\ti{{\tau}}_z} \left(\ti{g}_b(-|\ti{e}_z|)-\ti{g}_a(|\ti{e}_z|)\right).    \]
We look for $\ti{g}_i=\ti{g}_i^0+\varepsilon\ti{g}_i^1+...$, with  $\ti{N}[\ti{g}_i^k]=\int_0^Ln[\ti{g}_i^k/l]dz= 0$, for $i=1,2$ and $k\geq 1$.
The asymptotic analysis leads to \\
\noindent
{\em At the leading order}\;  For $i=1,2, \; \; \ti{{\Theta}}[\ti{g}_i^0]\ti{l}\ti{M}-\ti{g}_i^0=0$, 
which implies $\ti{g}_i^0=\frac{\ti{N}_i}{\gamma}\ti{l}\ti{M}$.\\
\noindent
{\em At order +1}
\begin{equation*}
\ti{{\Theta}}[\ti{g}_1]\ti{l}\ti{M}-\ti{g}_1=\frac{ \ti{\tau}_{ms}\chi^f(\ti{e}_z)}{ 2\ti{{\tau}}_z(\ti{e}_z)}(\ti{g}_2^0-\ti{g}_1^0)+
\ti{\tau}_{ms}\left( \ti{v}_{{x}}\partial_{\ti{x}} \ti{g}_1^0 -\ti{U}'(\ti{x})\partial_{\ti{v}_x}\ti{g}_1^0 \right).
\end{equation*}
A necessary condition of solvability  is that the integral of the right-hand-side with respect to $\ti{v}_x$ and $\ti{e}_x$ vanishes, 
which implies that
\[(\ti{N}_2-\ti{N}_1)\int_{\ti{e}_z}\int_{\ti{v}_x}\frac{\ti{\tau}_{ms}\chi^f}{2\ti{\tau}_z(\ti{e}_z)}\ti{l}(\ti{e}_z)\ti{M}d\ti{v}_xd\ti{e}_z=0,\]
and thus $\ti{N}_1=\ti{N}_2$. With this condition the hypothesis of Lemma 1 is satisfied by the right-hand-side and taking 
into account that $\ti{N}[\ti{g}_i^1]=0$ , we obtain 
$\ti{g}_1^i=\ti{v}_{{x}}\partial_{\ti{x}} \ti{g}_i^0 -\ti{U}'(\ti{x})\partial_{\ti{v}_x}\ti{g}_i^0 $. Finally, coming back to dimension variables
we conclude that $N_1=N_2=N_*/2$ are solution of a the same diffusion equation as $N_*$.

\subsubsection{Weak coupling between the two surface layer}

We assume now that
\begin{equation}
 \frac{n_f^*}{n^*}=\varepsilon \varepsilon_0 .  \label{smallnf}
\end{equation}
This assumption means that the ratio of free molecules is very small.
Under this hypothesis the ratio $\frac{n_f^* t_c^*}{n^*\tau_z^*}=\varepsilon$, so that the 
coupling of the two equations  is weak  (the coupling term is of order $\varepsilon$ while the collision term
is of order $1/\varepsilon$) and the dimensionless form of the system governing the evolution of $g_1$ and 
$g_2$ is
\begin{eqnarray}
\varepsilon \partial_{\ti{t}} \ti{g}_1+ \ti{v}_{{x}}\partial_{\ti{x}} \ti{g}_1 -\ti{U}'(\ti{x})\partial_{\ti{v}_x}\ti{g}_1&=& 
\Lambda[\ti{g}_1,\ti{g}_2] , \label{eqh-nc-1-adim-weak}\\
\varepsilon \partial_{\ti{t}} \ti{g}_2+ \ti{v}_{{x}}\partial_{\ti{x}} \ti{g}_2 -\ti{U}'(\ti{x})\partial_{\ti{v}_x}\ti{g}_2&=& 
\Lambda[\ti{g}_2,\ti{g}_1] , \label{eqh-nc-2-adim-weak}
\end{eqnarray}
where \[\Lambda[\ti{g}_a,\ti{g}_b]=\frac{1}{\varepsilon\ \ti{\tau}_{ms}}\left(\ti{{\Theta}}[\ti{g}_a]\ti{l}\ti{M}-\ti{g}_a\right)+ 
\frac{\varepsilon\ \chi^f(\ti{e}_z)}{ 2 \ti{{\tau}}_z(\ti{e}_z)} \left(\ti{g}_b(-|\ti{e}_z|)-\ti{g}_a(|\ti{e}_z|)\right). \]  
To study the diffusion limit of this system we integrate the two equations with respect to $\ti{e}_z$. It comes
\begin{eqnarray}
\varepsilon \partial_{\ti{t}} \ti{h}_1+ \ti{v}_{{x}}\partial_{\ti{x}} \ti{h}_1 -\ti{U}'(\ti{x})\partial_{\ti{v}_x}\ti{h}_1= 
\frac{1}{\varepsilon\ \ti{\tau}_{ms}}\left(\frac{\ti{N}_1}{\ti{\gamma}_x}\ti{M}_x-\ti{h}_1\right)+ 
\frac{\varepsilon}{ 2} \left(\ti{F}_2-\ti{F}_1\right)
 , \label{eqh-nc-1-adim-int}\\
\varepsilon \partial_{\ti{t}} \ti{h}_2+ \ti{v}_{{x}}\partial_{\ti{x}} \ti{h}_2 -\ti{U}'(\ti{x})\partial_{\ti{v}_x}\ti{h}_2= 
\frac{1}{\varepsilon\ \ti{\tau}_{ms}}\left(\frac{\ti{N}_2}{\ti{\gamma}_x}\ti{M}_x-\ti{h}_2\right)+ 
\frac{\varepsilon}{ 2} \left(\ti{F}_1-\ti{F}_2\right)
 , \label{eqh-nc-2-adim-int}
\end{eqnarray}
where $\ti{h}_i=\int_{\ti{e}_z}\ti{g}_id\ti{e}_z$,   $\ti{N}_i=\ti{N}[\ti{g}_i]=\ti{N}[\ti{h}_i]$, 
$\ti{F}_i=\int_{\ti{e}_z}\frac{\chi^f (\ti{e}_z)}{\ti{\tau}_z}\ti{g}_i(\ti{e}_z)d\ti{e}_z$. Unfortunately, this system is not closed
because, in general, we cannot write $\ti{F}_i$ as a function of $\ti{h}_i$. Nevertheless 
\begin{equation}
\mbox{if}\;
\ti{{g}}_i=\frac{\ti{h}_i}{\int\ti{l}\ti{M}_zd\ti{e}_z} \ti{l}(\ti{e}_z)\ti{M}_z(\ti{e}_z),\ \mbox{then}\  \ti{F}_i=\ti{c}\ \ti{h}_i, \;
\mbox{ where}\ \ti{c}=\frac{\int_{\ti{e}_z}\chi^f (\ti{e}_z) |\ti{e}_z|\ti{M}_z(\ti{e}_z)d\ti{e}_z}{\int\ti{l}\ti{M}_zd\ti{e}_z}.
\label{Fcond}
\end{equation}
But an asymptotic analysis from the dimensionless form of (\ref{eqh-nc-1-adim-weak}-\ref{eqh-nc-2-adim-weak})
 proves that $\ti{g}_i^0=(\ti{N}_i/\ti{\gamma})\ti{l}\ti{M}$ and 
$\ti{g}_i^1=(\ti{N}_i/\ti{\gamma})\ti{l}\ti{M}-\frac{\ti{\tau}_{ms}}{\gamma}\left( \partial_{\ti{x}}\ti{N}_i^0\ \ti{v}_{{x}}\ti{l}\ti{M}
- \ti{U}'(\ti{x})\partial_{\ti{v}_x}\ti{{g}}_i^0\right) $, so that
$\ti{g}_i^0+\varepsilon \ti{g}_i^1$ satisfies the assumption of (\ref{Fcond}). Thus we can write
(\ref{eqh-nc-1-adim-int}-\ref{eqh-nc-2-adim-int}) in the following form
\begin{eqnarray*}
\varepsilon \partial_{\ti{t}} \ti{h}_1+ \ti{v}_{{x}}\partial_{\ti{x}} \ti{h}_1 -\ti{U}'(\ti{x})\partial_{\ti{v}_x}\ti{h}_1&=& 
\frac{1}{\varepsilon\ \ti{\tau}_{ms}}\left(\frac{\ti{N}_1}{\ti{\gamma}_x}\ti{M}_x-\ti{h}_1\right)+ 
\frac{\varepsilon}{ 2}\ti{c} \left(\ti{h}_2-\ti{h}_1\right)+{\mathcal O}(\varepsilon^2)
 , \label{eqh-nc-1-adim-intbis}\\
\varepsilon \partial_{\ti{t}} \ti{h}_2+ \ti{v}_{{x}}\partial_{\ti{x}} \ti{h}_2 -\ti{U}'(\ti{x})\partial_{\ti{v}_x}\ti{h}_2&=& 
\frac{1}{\varepsilon \ti{\tau}_{ms}}\left(\frac{\ti{N}_2}{\ti{\gamma}_x}\ti{M}_x-\ti{h}_2\right)+ 
\frac{\varepsilon}{ 2}\ti{c} \left(\ti{h}_1-\ti{h}_2\right)+{\mathcal O}(\varepsilon^2),
  \label{eqh-nc-2-adim-intbis}
\end{eqnarray*}
where $\ti{\gamma}_x=\int_{\ti{v}_x}\ti{M}_x(\ti{v}_x)d\ti{v}_x= \sqrt{2\pi}$. 
The asymptotic analysis on this system follows the same ideas as in previous subsections and we get\\
\noindent
{\em At the leading order}
\begin{equation}
\ti{h}_i^0=\frac{\ti{N}_i^0}{\ti{\gamma}_x}\ti{M}_x,\;\; i=1,2, \;\; \mbox{with}\;\; \ti{N}_1^0+\ti{N}_2^0=\ti{N}_*^0.   \label{diff-free-z1}
\end{equation}

\noindent
{\em At order +1}

\begin{eqnarray}
 \ti{h}_i^1 &=& \frac{\ti{N}_i^1}{\ti{\gamma}_x}\ti{M}_x-\ti{\tau}_{ms}\left( \partial_{\ti{x}}\frac{\ti{N}_i^0}{\ti{\gamma}_x}\ \ti{v}_{{x}}\ti{M}_x
- \ti{U}'(\ti{x})\partial_{\ti{v}_x}\ti{h}_i^0\right).
 \label{diff-free-o1}
\end{eqnarray}
\noindent
{\em At order +2}
\begin{eqnarray*}
 \partial_{\ti{t}} \ti{h}_1^0+ \ti{v}_{{x}}\partial_{\ti{x}} \ti{h}_1^1-\ti{U}'(\ti{x})\partial_{\ti{v}_x}\ti{h}_1^1 &=& 
\frac{1}{ \ti{\tau}_{ms}}\left(\frac{\ti{N}[\ti{h}_1^2]}{\ti{\gamma}_x}\ti{M}_x-\ti{h}_1^2\right)+ 
\frac{\ti{c}}{ 2} \left(\ti{h}_2^0-\ti{h}_1^0\right)
 , \label{eqh-nc-1-o2}\\
\partial_{\ti{t}} \ti{h}_2^0+ \ti{v}_{{x}}\partial_{\ti{x}} \ti{h}_2^1-\ti{U}'(\ti{x})\partial_{\ti{v}_x}\ti{h}_2^1 &=& 
\frac{1}{ \ti{\tau}_{ms}}\left(\frac{N[\ti{h}_2^2]}{\ti{\gamma}_x}\ti{M}_x-\ti{h}_2^2\right)+ 
\frac{\ti{c}}{ 2} \left(\ti{h}_1^0-\ti{h}_2^0\right).
 \label{eqh-nc-2-o2}
\end{eqnarray*}
With this form of the equations, we can use Lemma 1 and
$\ti{h}_1^2$ and $\ti{h}_2^2$, can be defined from those relations if and only if the following solvability conditions holds
\begin{eqnarray*}
\int_{\ti{v}_x} \left( \partial_{\ti{t}} \ti{h}_1^0+ \ti{v}_{{x}}\partial_{\ti{x}} \ti{h}_1^1-\ti{U}'(\ti{x})\partial_{\ti{v}_x}\ti{h}_1^1-
\frac{\ti{c}}{ 2} \left(\ti{h}_2^0-\ti{h}_1^0\right) \right) d\ti{v}_x & =& 0,
  \label{eqh-nc-1-o2-int}\\
\int_{\ti{v}_x} \left( \partial_{\ti{t}} \ti{h}_2^0+ \ti{v}_{{x}}\partial_{\ti{x}} \ti{h}_2^1-\ti{U}'(\ti{x})\partial_{\ti{v}_x}\ti{h}_2^1-
\frac{\ti{c}}{ 2} \left(\ti{h}_1^0-\ti{h}_2^0\right) \right) d\ti{v}_x &=& 0,
  \label{eqh-nc-2-o2-int}
\end{eqnarray*}
After inserting (\ref{diff-free-z1}-\ref{diff-free-o1}) in the above relations and coming back in dimension variables we finally obtain the following result

\begin{proposition}
Under the hypothesis 
(\ref{simplepot1}-\ref{simplepot2}-\ref{nocoll}-\ref{tauconst}-\ref{simbound}-\ref{smallnf}), 
the solutions $g_i,\ i=1,2$ of (\ref{eqh-nc-1}-\ref{eqh-nc-2} ) formally converge as 
$\varepsilon(={t}_c^*/{t}_d^*={\tau}_{ms}^*/t_c^*) \rightarrow 0$
to 
\[ \left( N_i(t,x)/\gamma \right)l(e_z)\ M(v_x,e_z), i=1,2\]
 where the functions $N_i(t,x), i=1,2$ are solutions 
of the following system of diffusion equations
\begin{eqnarray*}
\partial_{{t}}{N}_1-{{D}_0^{(n)}\partial^2_{{x}^2}{N}_1}-\tau_{ms}\partial_x\left(\frac{U'(x)}{m}N_1 \right) &=& {c}\ ({N}_2-{N}_1),\\
\partial_{{t}}{N}_2-{{D}_0^{(n)}\partial^2_{{x}^2}{N}_2}-\tau_{ms}\partial_x\left(\frac{U'(x)}{m}N_2 \right) &=& {c}\ ({N}_1-{N}_2),
\end{eqnarray*}
where
\begin{eqnarray*}
D_0^{(n)}&=& \frac{\tau_{ms}}{\gamma}  \langle\langle  {{v}_x}^2 {M} \rangle\rangle, \\
\gamma = \langle\langle M \rangle\rangle, \ {M}(v_x,e_z)&=& e^{-m(v_x^2+e_z^2)/2kT}, \ 
c=c(W_m) = \frac{\int_{e_z}\chi^f(e_z)|e_z|M_z(e_z)de_z}{\int_{e_z}l(e_z)M_z(e_z)de_z}. 
\end{eqnarray*}
\end{proposition}

\begin{remark}
From this model, thanks to convenient rescalings, we can recover some of the diffusion models obtained above
\begin{enumerate}
\item In the limit of a large  $W_m$ the number of free molecules tends to $0$ and so does 
the coefficient $c=c(W_m) = (\int_{e_z}\chi^f(e_z)|e_z|M_z(e_z)de_z)/(\int_{e_z}l(e_z)M_z(e_z)de_z)$. 
Thus we recover in this limit the result of proposition 3, where we found a diffusion
equation for the trapped molecules assuming that the free molecules could be neglected. In the present configuration, 
we obtain, in the limit of a large $W_m$ two independent diffusion equations for $N_1$ and $N_2$.\\
\item In the limit of a small $\tau_{z}$, then $N_2=N_1$, and we recover the diffusion model of the moderate coupling regime.
\end{enumerate}
\end{remark}
\section{Conclusion}
We have presented the formal derivation of a hierarchy of models describing a gas flow in the vicinity of a wall,
using scaling and systematic asymptotic analysis. Following the ideas introduced in \cite{Borman},\cite{BKP2},\cite{Kry1},
\cite{Kry2},\cite{Kry3},\cite{Kry4},\cite{BBK},
the influence of the wall is taken into account through Van der Wall forces acting
on the gas molecules and through a relaxation of the gas molecules by the substrate.\\

In this paper we made some simplifying assumptions: we assumed that the molecules move in a 2D plane,
we considered the case where the intermolecular collisions can be neglected
and we assumed that the interaction potential has a simplified structure.\\
 With those assumptions we derived a multiphase model 
involving a classical kinetic equation for the bulk flow coupled with two one-dimensional kinetic equations modeling 
the trapped and free molecules inside the surface layer. This one-dimensional kinetic model can be interpreted as a
non-local boundary condition for the bulk flow.\\
Then, assuming that the interaction potential is rapidly
oscillating in the direction parallel to the solid surface, an averaged mesoscopic kinetic model is obtained by 
homogeneization. \\
Finally, in the limit of a small relaxation time, we derived from the multiphase kinetic model 
diffusion models for the surface molecules. In a first step we assume that the free molecules can be neglected and
in a second step we consider a narrow channel constituted by two surface layers. Then we took into account the trapped
and free molecules in the channel and we derived 
several diffusion models according to the (strong, moderate or weak) coupling of the two surface layers. \\
The extension of those models for more general interaction potentials and for
collisional flows will be studied in forthcoming papers. \\

\section*{Acknowledgments}
Part of this work has been conducted during while K.A. was visiting the
Institut de Math\'ematiques de Toulouse, under the auspices of the foundation 
``Sciences et Technologies pour l'A\'eronautique et l'Espace'' under the grant
``Plasmax'' (RTRA-STAE/2007/PF/002).  
K.A. and P.D. express their cordial thanks to the Isaac Newton
Institute for Mathematical Sciences for its hospitality during the preparation
of the present paper. This work was partially supported by the Grant-in-Aid
for Scientific Research No.~20360046 from JSPS. P.D. has been supported  
by the Marie Curie Actions of the European 
Commission in the frame of the DEASE project (MEST-CT-2005-021122)


\medskip

\end{document}